\documentclass[12pt]{amsart}

\usepackage{amsfonts,amssymb,stmaryrd,amscd,amsmath,latexsym,amsbsy}
\usepackage[breaklinks]{hyperref}

\newtheorem{theorem}{Theorem}[section]
\newtheorem{lemma}[theorem]{Lemma}
\newtheorem{proposition}[theorem]{Proposition}
\newtheorem{corollary}[theorem]{Corollary}
\theoremstyle{definition}

\newtheorem{example}[theorem]{Example}

\newtheorem{question}[theorem]{Question}
\newtheorem{conjecture}[theorem]{Conjecture}
\newtheorem{remark}[theorem]{Remark}

\newcommand{\BQ}{{\mathbb Q}}
\newcommand{\FP}{{\rm FPdim}}
\newcommand{\uRep}{\underline{\mbox{Rep}}}
\newcommand{\cA}{{\mathcal A}}
\newcommand{\BF}{{\mathbb F}}
\newcommand{\T}{{\mathcal T}}
\newcommand{\BC}{{\mathbb C}}
\newcommand{\X}{\mathcal{X}}
\newcommand{\s}{\mathfrak{s}}

\newcommand{\Hom}{\text{Hom}}

\newcommand{\Aut}{\text{Aut}}

\newcommand{\Rep}{\text{Rep}}

\newcommand{\diag}{\text{diag}}

\newcommand{\g}{\mathfrak{g}}

\newcommand{\h}{\mathfrak{h}}

\newcommand{\C}{\mathcal{C}}

\newcommand{\ben}{\begin{enumerate}}
\newcommand{\een}{\end{enumerate}}

\newcommand{\be}{{\bf 1}}

\theoremstyle{plain}

\newtheorem*{sol}{Solution}

\theoremstyle{definition}

\theoremstyle{remark}

\newcommand{\solu}[1]{\begin{sol}{\bf (\ref{#1})}}

\def\g{\mathfrak{g}}

\def\C{\mathcal{C}}
\def\D{\mathcal{D}}

\def\Aut{\mathop{\mathrm{Aut}}\nolimits}
\def\h{\mathfrak{h}}

\def\Z{\mathbb{Z}}

\def\Hom{\mathrm{Hom}}

\def\Vec{\mathrm{Vec}}

\def\Ver{\mathrm{Ver}}

\def\Rep{\mathop{\mathrm{Rep}}\nolimits}

\numberwithin{equation}{section}

\begin{document}

\title{Asymptotic properties of tensor powers in symmetric tensor categories}

\author{Kevin Coulembier}
\address{School of Mathematics and Statistics, University of Sydney, Australia}
\email{kevin.coulembier@sydney.edu.au}
\author{Pavel Etingof}
\address{Department of Mathematics, MIT, Cambridge, MA USA 02139
}
\email{etingof@math.mit.edu}
\author{Victor Ostrik}
\address{Department of Mathematics,
University of Oregon, Eugene, OR 97403, USA}
\email{vostrik@math.uoregon.edu}
\maketitle

\centerline{\bf To Claudio Procesi on his 80-th birthday with admiration}

\tableofcontents

\begin{abstract} Let $G$ be a group and $V$ a finite dimensional representation 
of $G$ over an algebraically closed field $\bold k$ of characteristic $p>0$. Let $d_n(V)$ 
be the number of indecomposable summands of $V^{\otimes n}$ 
of nonzero dimension mod $p$. It is easy to see that there exists a limit $\delta(V):=\lim_{n\to \infty}d_n(V)^{1/n}$, which is positive (and $\ge 1$) iff $V$ has an indecomposable summand 
of nonzero dimension mod $p$. We show that in this case 
the number 
$$
c(V):=\liminf_{n\to \infty} \frac{d_n(V)}{\delta(V)^n}\in [0,1]
$$ 
is strictly positive and 
$$
\log (c(V)^{-1})=O(\delta(V)^2),
$$ 
and moreover this holds for any symmetric tensor category over $\bold k$ of moderate growth. Furthermore, we conjecture that in fact 
$$
\log(c(V)^{-1})=O(\delta(V))
$$ 
(which would be sharp), and prove this for $p=2,3$; in particular, for $p=2$ we show that $c(V)\ge 3^{-\frac{4}{3}\delta(V)+1}$. The proofs are based on the characteristic $p$ version of Deligne's theorem for symmetric tensor categories obtained in \cite{CEO}. 
We also conjecture a classification of semisimple symmetric tensor categories of moderate growth which is interesting in its own right and implies the above conjecture for all $p$, and illustrate this conjecture by describing the semisimplification of the modular representation category 
of a cyclic $p$-group. Finally, we study the asymptotic behavior of the decomposition of $V^{\otimes n}$ in characteristic zero using Deligne's theorem and the Macdonald-Mehta-Opdam identity. 
 \end{abstract} 

 \section{Introduction} 

\subsection{Growth of the non-negligible part of a tensor power} 
$\quad$ Asymptotic properties of tensor powers of a modular representation (or, more generally, an object of a symmetric tensor category) are an interesting and mysterious subject, about which rather little is known. It has recently been studied in the papers \cite{B2, BS,CEO,EK,COT}. The goal of this paper is to continue this study. 

Let $\bold k$ be an algebraically closed field. Let $\mathcal C$ 
be a symmetric tensor category over $\bold k$. Then each $X\in \C$ uniquely decomposes into a direct sum of indecomposable objects. We say that $X$ is {\it negligible} if the categorical dimension of every indecomposable summand of $X$ is zero; in particular, an indecomposable $X$ is negligible iff $\dim X=0$. For an object $V\in \C$ let $d_n(V)$ be the number of non-negligible indecomposable summands of $V^{\otimes n}$ with multiplicities. Then clearly $d_{n+m}(V)\ge d_n(V)d_m(V)$, but otherwise the behavior of this sequence is mysterious and interesting to study. 

In particular, assume that $\mathcal C$ is of {\it moderate growth}, i.e., the length of $V^{\otimes n}$  grows at most exponentially with $n$ (hence so does $d_n(V)$). For instance, we may take $\mathcal C$ to be a {\it super-Tannakian} category, which is the only possibility if ${\rm char}(\bold k)=0$ by Deligne's theorem \cite{D1}.\footnote{Recall that a super-Tannakian category is the category ${\rm Rep}(G,z)$ of finite dimensional representations of an affine supergroup scheme $G$ on which some element $z\in G(\bold k)$, $z^2=1$ defining the parity automorphism of $G$ acts by the parity operator. Deligne's theorem \cite{D1} states that if ${\rm char}(\bold k)=0$ then $\C$ has moderate growth if and only if it is super-Tannakian.} Then by Fekete's lemma (\cite{B1}, Lemma 1.6.3) there exists a non-negative limit 
$$
\delta(V):=\lim_{n\to \infty}d_n(V)^{1/n},
$$ 
and 
$d_n(V)\le \delta(V)^n$ for all $n\ge 1$.  

Let ${\rm Green}(\mathcal C)$ 
be the Green ring of $\mathcal C$ (freely spanned by the classes of indecomposable objects of $\C$) 
and $N\subset {\rm Green}(\mathcal C)$ be the subgroup spanned by the negligible objects. Then $N$ is an ideal, and by 
Deligne's theorem \cite{D1} in characteristic zero and \cite{CEO} in characteristic $p>0$, 
$\delta$ extends to a character  
$$
\delta: {\rm Green}(\mathcal C)/N\to \Bbb R
$$ 
which sends basis elements to 
positive numbers (obviously $
\ge 1$). Moreover, Deligne's theorem implies that if ${\rm char}(\bold k)=0$ then $\delta(V)\in \Bbb Z_{\ge 0}$ and for Tannakian categories (i.e., representation categories of affine group schemes) $\delta(V)=\dim V$. On the other hand,
for characteristic $p>0$ it is shown in \cite{CEO} that $\delta(V)=\sum_{k=1}^{p-1} m_k[k]_q$, where $m_k\in \Bbb Z_{\ge 0}$, $q=e^{\pi i/p}$ and $[k]_q=\frac{q^k-q^{-k}}{q-q^{-1}}$; in particular, 
$\delta(V)$ is an integer for $p=2,3$ but may be irrational for $p\ge 5$. 

If $V$ is non-negligible (i.e., $\delta(V)\ne 0$), let 
$$
c_n(V):=\frac{d_n(V)}{\delta(V)^n}.
$$
Thus $\lim_{n\to \infty}c_n(V)^{1/n}=1$ and $c_n(V)\le 1$. The goal of this paper is to understand the behavior of the sequence $c_n(V)$ (and thereby $d_n(V)$) in more detail. In particular, we are 
interested in 
$$
c(V):=\liminf_{n\to \infty}c_n(V).
$$

\subsection{Growth in characteristic zero} 
For example, let ${\rm char}(\bold k)=0$ and $\C={\rm Rep}_{\bold k}(G)$, where $G$ is a finite group, so $\delta(V)=\dim V$. Let $V\in \C$ be a faithful representation of $G$. We have 
$$
\sum_{L\in {\rm Irr}G}[V^{\otimes n}:L]\dim L=(\dim V)^n,
$$ 
where ${\rm Irr}G$ is the set of isomorphism classes of irreducible representations of $G$ over $\bold k$. Thus
$$
c(V)=\liminf_{n\to \infty}\frac{\sum_{L\in {\rm Irr}G}[V^{\otimes n}:L]}{(\dim V)^n}\ge \frac{1}{\max_{L \in {\rm Irr}G}\dim L}\ge \frac{1}{|G/N|},
$$
where $N$ is a normal abelian subgroup of $G$ of maximal size 
(as $L$ is generated by an eigenvector of $N$).

Recall that the Jordan-Schur bound 
for complex representations of finite groups (see \cite{Co} and references therein) yields 
\begin{equation}\label{jsb}
|G/N|=O((\dim V+1)!).
\end{equation}
So we get 
$$
c(V)\ge \frac{K}{(\delta(V)+1)!}
$$
for a fixed $K>0$, i.e., using Stirling's formula
\begin{equation}\label{bou2}
\log (c(V)^{-1})\le \delta(V)(\log \delta(V)-1)+\tfrac{1}{2}\log \delta(V)+O(1).
\end{equation} 

To the contrary, for infinite groups and for supergroups $G$ in characteristic zero we may have $c(V)=0$. Namely,  let $\overline{\mathcal C}$ be the semisimplification of $\mathcal C$ (see 
\cite{AK, EK,EO} and references therein). Then by Deligne's theorem, $\overline{\mathcal C}$ is of the form ${\rm Rep}(G_\C,z)$ for a reductive pro-algebraic supergroup $G_{\mathcal C}$, and $d_n(V)$ 
is the length of $\overline{V}^{\otimes n}$, where $\overline{V}$ is the image of $V$ in $\overline{\mathcal C}$. It follows that $\delta(V)=\dim_{\bold k}\overline V$ is an integer, and the behavior of $c_n(V)$ is easily understood 
using the classical representation theory of connected reductive groups (as well as supergroups ${\rm SOSp}(1,2n)$), see Section 2. In  particular, if the image of ${\rm Lie}G_{\mathcal C}$ in ${\rm End}_{\bold k}\overline V$ is non-abelian then 
$$
K_1n^{-b/2}\le c_n(V)\le K_2n^{-b/2}
$$ 
for some $K_1,K_2>0$ and integer $b\ge 1$, so $c_n(V)\to 0$ as $n\to \infty$, hence $c(V)=0$. 
For example, if $\mathcal C={\rm Rep}SL_2(\Bbb C)$ and $V=\Bbb C^2$ then $\delta(V)=2$, while 
\begin{equation}\label{coin}
d_n(V)=\binom{n}{[n/2]}\sim \frac{2^n}{\sqrt{\pi n/2}},
\end{equation}
so $c_n(V)=2^{-n}\binom{n}{[n/2]}\sim \frac{1}{\sqrt{\pi n/2}}\to 0,\ n\to \infty$. This happens because simple summands occurring in $V^{\otimes n}$ have unbounded dimension. 

\subsection{Growth in characteristic $p>0$: the main theorem and conjectures} 
In positive characteristic, however, the situation is different. Namely, \cite{CEO}, Corollary 8.13 implies that in this case if $V$ is non-negligible then $c(V)>0$. Moreover,  
we have the following theorem (proved in Section 3), which is our main result. 

\begin{theorem}\label{th1} Let ${\rm char}(\bold k)=p>0$. Let $V$ be a non-negligible object in a symmetric tensor category $\C$ of moderate growth over $\bold k$ (i.e., such that $\delta(V)\ne 0$). Let $\lambda(V):=\log(c(V)^{-1})$. Then for $p=2,3$ 
$$
\lambda(V)\le a_p\delta(V),
$$
while for $p\ge 5$
$$
\lambda(V)\le a_p\delta(V)+\frac{\pi\log 2}{2}(p-2)\delta(V)^2
$$ 
for some $a_p>0$. Specifically,  
we may take 
$$
a_2=\frac{4\log 3}{3}=1.464...,\ a_3=\frac{\log 3}{3}+\log 7=2.312...
$$ 
\end{theorem} 

The same bound applies to $c_n(V)$, see Subsection \ref{cnv}. 

We note that Theorem \ref{th1} is new already for representations of a finite group (or more generally, affine group scheme); in fact, it is new whenever the group has wild representation type. 

The bound of Theorem \ref{th1} is not expected to be sharp for any $p\ge 5$, and we expect that the last summand on the right hand side is not really needed. Namely, we have

\begin{conjecture} \label{co1} For all $p$ 
$$
\lambda(V)\le A_p\delta(V).
$$ 
for some $A_p>0$.
\end{conjecture} 

Note that by Theorem \ref{th1} this conjecture holds for $p=2,3$. 

\subsection{Conjectural classification of semisimple symmetric tensor categories and its application to Conjecture \ref{co1}} 
We will see that Conjecture \ref{co1} follows from the following conjecture about affine group schemes
in the Verlinde category ${\rm Ver}_p$, $p\ge 5$ (for basics on Lie theory in ${\rm Ver}_p$, see \cite{E},\cite{V}). Let $\pi:=\pi_1({\rm Ver}_p)$ be the fundamental group of ${\rm Ver}_p$; it acts canonically on every object of ${\rm Ver}_p$. Let $G$ be an affine group scheme in ${\rm Ver}_p$ equipped with a homomorphism $\phi: \pi\to G$ which induces the canonical action of $\pi$ on $O(G)$. Denote by $\Rep(G,\phi)$ the category of representations of $G$ in ${\rm Ver}_p$ such that $\pi$ acts canonically (via $\phi$). We say that $G$ is {\it linearly reductive} if the category ${\rm Rep}(G,\phi)$ is semisimple. 

Let $G$ be a finite group scheme in ${\rm Ver}_p$ and $\g:={\rm Lie}G$. Then $\g$ is contained in the group algebra $O(G)^*$. We will say that $G$ is {\it of height $1$} if $\g$ generates $O(G)^*$; in other words, the Frobenius twist of $G$ is trivial. 

Let FPdim denote the Frobenius-Perron dimension of an object of a finite tensor category.

\begin{conjecture}\label{eveneasier} There exists $M_p>0$ such that if $G$ is a linearly reductive height $1$ subgroup scheme of $GL(Y)$, $Y\in {\rm Ver}_p$, then 
$$
{\rm FPdim}\g\le M_p{\rm FPdim}Y.
$$
\end{conjecture} 

Note that this is stronger than the trivial bound ${\rm FPdim}\g\le ({\rm FPdim}Y)^2$, and it fails for height $1$ group schemes which are not linearly reductive (e.g. $G=GL(Y)_{(1)}$). 

Clearly, it suffices to establish Conjecture \ref{eveneasier} when $Y$ is a simple $G$-module: then in general if $Y=\oplus_i Y_i$ 
where $Y_i$ are simple and $\g_i$ is the image of $\g$ in ${\mathfrak{gl}}(Y_i)$, then 
$$
{\rm FPdim}\g\le \sum_i {\rm FPdim}\g_i\le M_p\sum_i {\rm FPdim}Y_i=M_p{\rm FPdim}Y.
$$

We will say that a linearly reductive $G$ of height $1$ is {\it simple} if $\g$ is a simple $G$-module and $G\ne (\Bbb Z/p)^\vee$. It is easy to see that any linearly reductive group scheme of height $1$ is a direct product of simple ones and copies of $(\Bbb Z/p)^\vee$. Thus Conjecture \ref{eveneasier} follows from 

\begin{conjecture}\label{easier} There exists $D_p\ge 1$ such that 
for any simple linearly reductive $G$ in ${\rm Ver}_p$ we have 
${\rm FPdim} \g\le D_p$. 
\end{conjecture}

Indeed, if so then in general $\g$ has $m\ge \frac{{\rm FPdim}\g}{D_p}$ direct factors $\g_i$, $1\le i\le m$, which are simple or isomorphic to $(\Bbb Z/p)^\vee$, so if $Y$ is simple (and faithful; which implies in particular that at most one factor of $G$ is isomorphic to $(\Bbb Z/p)^\vee$) then $Y=\boxtimes_{i=1}^m Y_i$ where $Y_i$ is a simple $\g_i$-module. Hence 
$$
{\rm FPdim}Y=\prod_{i=1}^m {\rm FPdim}(Y_i)\ge (q+q^{-1})^{m-1}\ge \frac{1+(q+q^{-1}-1)m}{q+q^{-1}}.
$$
Thus we may take $M_p=\frac{(q+q^{-1})D_p}{q+q^{-1}-1}$. 

In fact, in Section 4 we will make a much stronger conjecture (Conjecture \ref{mai}), which provides a classification of semisimple symmetric tensor categories (and thereby of linearly reductive group schemes in ${\rm Ver}_p$), and thus is interesting in its own right. This conjecture was proposed by the third author around 2017, and since then there has been mounting evidence that it should be true. For $p=2,3$ Conjecture \ref{mai} is proved in \cite{CEO} (in which case there are no simple linearly reductive group schemes). In \cite{EKO} it will be shown that Conjecture \ref{mai} also holds for $p=5$ (with the only simple linearly reductive group scheme being the identity component $\pi_+$ of $\pi$). Thus the results of \cite{EKO} will imply that all the above conjectures hold for $p=5$. However, starting with $p=7$, they remain open. 

 \subsection{Cyclic groups} In Section 5 we give an example illustrating Conjecture \ref{mai}. Namely we
consider the semisimplification of the category of finite dimensional representations of the cyclic group $\Bbb Z/{p^n}$ of
order $p^n$ (where $p={\rm char}(\bold k)$) and show that the resulting category does satisfy
Conjecture \ref{mai}. Recall that in the case $n=1$ the semisimplification of $\Rep(\Bbb Z/p)$ is the Verlinde
category ${\rm Ver}_p$. Using the classical results of Green from 
\cite{Gr} 
we show that for $n>1$
the semisimplification of $\Rep(\Bbb Z/{p^n})$ factorizes into a product of one copy of ${\rm Ver}_p$ and $n-1$ copies of a certain category $\D$. In the setup of the conjectures in Section 4, the category $\D$
is associated with the Lie supergroup ${\rm SOSp}(1|2)$ or, equivalently, with the group ${\rm SO}(p-1)$. 
The same result is valid for the representation category 
of the Hopf algebra $\bold k[x]/(x^{p^n})$ where $x$ is a primitive element, and 
more generally for any cocommutative Hopf algebra structure on 
$\bold k[x]/(x^{p^n})$ (there are many such structures related to truncations of 1-dimensional formal group laws). 
 
 \subsection{Organization of the paper} 
The organization of the paper is as follows. In Section 2 we discuss the characteristic zero case. In particular, we apply the Macdonald-Mehta-Opdam identity to give a formula 
for the leading coefficient of the asymptotics of 
$$
d_n(V,s)=\sum_L [V^{\otimes n}:L](\dim_{\bold k} L)^s
$$ 
(where the sum is over indecomposable objects). In Section 3 we discuss the case of characteristic $p$ and prove Theorem \ref{th1}. In Section 4 we state and discuss Conjecture \ref{mai}. 
Finally,  in Section 5 we illustrate this conjecture by the cyclic group example.

\subsection{Acknowledgements} P. E. thanks D. Benson, G. McNinch, P. Papi  and G. Robinson and V.O. thanks J. Wood for useful discussions.
 P. E.'s work was partially supported by the NSF grant DMS - 1916120.
 K. C.'s work was partially supported by the ARC grant DP210100251.

\section{Characteristic zero}

\subsection{Finite groups in the non-modular case} 
Let $G$ be a finite group of order nonzero in $\bold k$, and $V$ a finite dimensional representation of $G$.  

\begin{proposition}\label{bou} (i) Suppose that $V$ is faithful and 
any non-trivial $g\in G$ acts on $V$ by a non-scalar. 
Then the limit $c(V)=\lim_{n\to \infty}c_n(V)$ exists and is given by the formula 
$$
c(V)=\frac{1}{|G|}\sum_{L \in {\rm Irr}G}\dim_{\bold k} L.
$$ 

(ii) In this case 
$$
c(V)\le \sqrt{k(G)/|G|},
$$
where $k(G)$ is the number of conjugacy classes in $G$. 
\end{proposition}

\begin{proof}
(i) Without loss of generality we may assume that $\bold k=\Bbb C$. 
In this case our assumption on the faithful representation $V$ is equivalent to the assumption that 
for any non-trivial element $g\in G$  one has
$$
|\chi_V(g)|<\dim_{\bold k} V,
$$ 
where $\chi_V$ is the character of $V$. Thus we
have $(\frac{\chi_V}{\dim_{\bold k} V})^n=\delta_1+O(\lambda^n)$ 
for some $0<\lambda<1$, where $\delta_1$ is the delta function of the unit element. 
Thus 
$$
\sum_{L\in {\rm Irr}G}((\tfrac{\chi_V}{\dim_{\bold k} V})^n,\chi_L)\to  \frac{1}{|G|}\sum_{L \in {\rm Irr}G}\dim_{\bold k} L,\ n\to \infty, 
$$ 
as claimed.

(ii) By the Cauchy-Schwarz inequality and Maschke's theorem
$$
(\sum_{L \in {\rm Irr}G}\dim_{\bold k} L)^2\le \sum_{L \in {\rm Irr}G}1\cdot \sum_{L \in {\rm Irr}G}(\dim_{\bold k} L)^2=| {\rm Irr}G|\cdot |G|=k(G)\cdot |G|.
$$
Thus (ii) follows from (i).  
\end{proof} 

\begin{remark} 1. The assumption of Proposition \ref{bou} is satisfied for a faithful $V$ if $G$ has trivial center, or, say, when $V$ is  the direct sum of a trivial and a non-trivial representation. On the other hand, it fails if $G$ has non-trivial center and $V$ is an irreducible representation of $G$ with central character of order $r>1$. In  this case, $c_n(V)$ will in general have $r$ subsequential limits depending on $n$ mod $r$. For example, if $G$ is a non-abelian group of order $8$ and $V$ is the 2-dimensional irreducible representation, then $c_n(V)$
oscillates between $1$ and $\frac{1}{2}$. This explains why we focus our attention on $\liminf_{n\to \infty} c_n(V)$ (considering also that we always have $c_n(V)\le 1$).  

2. Proposition \ref{bou} shows that the bound in Conjecture \ref{co1} is sharp in the sense that it can't possibly be better than linear already for finite groups. E.g. let $G\ne 1$ be a finite group of order coprime to $p$ with trivial center, and 
$E$ be a faithful representation of $G$ over $\bold k$ of dimension $r$. Then 
$V_n:=E^n$ is a faithful representation of $G^n$ of dimension $nr$, i.e., $\delta(V_n)=nr$. Thus by Proposition \ref{bou} $c(V_n)\le (k(G)/|G|)^{n/2}$, so 
$$
\lambda(V_n)\ge n\log(|G|/k(G)),
$$ 
which is linear in $\delta(V_n)=nr$. 
\end{remark} 

\begin{example} Let $G=S_m$. Then all irreducible representations of $G$ are of real type, so by the Frobenius-Schur theorem the sum of their dimensions equals the number of involutions of $G$. Thus if $m\ge 3$  and $V$ is a faithful complex representation of $G$ then
$$
c(V)=\sum_{k=0}^{[m/2]} \frac{1}{(m-2k)!k!2^k}.
$$ 
Since $(m-2k)!k!2^k$ is minimal when $m-2k$ is about $\sqrt{m}$, using Stirling's formula we get  
\begin{equation}\label{bd}\lambda(V)=\tfrac{1}{2}m(\log m+O(1)),\ m\to \infty.
\end{equation}
This, in particular, shows that bound \eqref{bou2} is sharp 
up to a pre-factor $\frac{1}{2}$, since $S_m\subset GL(m-1)$. 
\end{example}

\subsection{The asymptotic formula} 
Let $\bold k$ have characteristic zero and 
$\C$ be a symmetric tensor category over $\bold k$. Recall that we are interested in the asymptotics 
 of the number $d_n(V)$ of non-negligible indecomposable summands of $V^{\otimes n}$. 
 In fact, we will consider an even more general problem -- the asymptotics of 
$$
d_n(V,s):=\sum_{L}[V^{\otimes n}:L]\delta(L)^s,
$$ 
where $s\in \Bbb C$ and the sum is over non-negligible indecomposables; e.g., $d_n(V,0)=d_n(V)$, while $d_n(V,1)=\delta(V)^n$, so this is a trivial case. Note that if $\C={\rm Rep}(G)$ where $G$ is an affine pro-algebraic group 
then there are no negligible summands and 
$$
d_n(V,s):=\sum_{L}[V^{\otimes n}:L](\dim_{\bold k}L)^s.
$$

Our main result in this section is the following theorem. 

\begin{theorem}\label{th3} For any $s\ge -1$, we have
$$
C_{V,1}(s)\le \frac{d_n(V,s)}{n^{\frac{m_V}{2}(s-1)}\delta(V)^n}\le C_{V,2}(s),
$$ 
where $m_V\in \Bbb Z_{\ge 0}$ and $C_{V,i}(s)>0$ are suitable constants. 
\end{theorem} 

To prove Theorem \ref{th3}, consider the semisimplification $\overline\C$ of $\C$ and its tensor subcategory $\overline \C_V$ generated by $V$. Then by Deligne's theorem \cite{D1}, $\overline\C_V$ is the 
equivariantization under a finite group of the Deligne tensor product of representation categories of connected reductive groups and supergroups $SOSp(1,2n)$. Thus it is easy to see that it suffices to prove the theorem for such a tensor product. Moreover, it is known that 
the tensor product rule for $SOSp(1,2n)$ is the same as for $SO(2n+1)$ (\cite{RS}), so 
it, in fact, suffices to prove the theorem for connected reductive groups. 
This is done in the following subsections. 

\subsection{Tensor powers of a representation of a connected reductive group} 

Let $G$ be a connected reductive algebraic group over $\bold k$ of rank $r$ 
with set of positive roots $R_+$, and 
$V$ a finite dimensional faithful algebraic representation of $G$.

\begin{theorem}\label{th4} For ${\rm Re}s\ge -1$, 
\begin{equation}\label{asy}
\frac{d_n(V,s)}{(\dim V)^n}\sim C_V(s)n^{\frac{(s-1)|R_+|}{2}},\ n\to \infty, 
\end{equation} 
where $C_V(s)$ is a holomorphic function of $s$. 
\end{theorem} 

\begin{proof} 
The proof is very similar to the proof in the case 
$s=0$ given in \cite{Bi,TZ,PR}, but we repeat the arguments for reader's convenience. 

We will use the following (well known) version of the central limit theorem of probability theory. 
Let $E$ be a real vector space with a finite spanning set $S$. Let 
$p: S\to \Bbb R_{>0}$ be a function such that $\sum_{\lambda\in S} p_\lambda=1$ and $\sum_{\lambda\in S} p_\lambda \lambda=0$. Then 
$Q:=\sum_{\lambda\in S} p_\lambda \lambda^2$ defines a positive definite quadratic form on $E^*$: 
$$
Q(f):=\sum_{\lambda\in S} p_\lambda f(\lambda)^2.
$$ 

\begin{theorem}\label{CLT}
Let $\xi_1,...,\xi_N$ be 
independent $S$-valued random variables which take value $\lambda$ with probability $p_\lambda$. Then 
the $E$-valued random variable $\eta_N:=\frac{\xi_1+...+\xi_N}{\sqrt{N}}$ 
converges as $N\to \infty$ (in distribution) to a normally distributed random variable 
$\eta(Q)$ with distribution density 
$$
p(v):=(2\pi)^{-\frac{\dim E}{2}}(\det Q)^{-\frac{1}{2}}\exp(-\tfrac{Q^{-1}(v)}{2}).
$$
\end{theorem} 

Let us now apply Theorem \ref{CLT}. We take for $E$ the dual space $\h_{\Bbb R}^*$ to a split real Cartan subalgebra $\h_\Bbb R$ in $\g={\rm Lie}G$, for $S$ the set of weights of the representation $V$, and $p_\lambda:=\frac{\dim V[\lambda]}{\dim V}$ for a weight $\lambda$. Then for $h\in \h$ we have $Q_V(h):=\frac{{\rm Tr}_V(h^2)}{\dim V}$. For brevity we will denote $Q_V$ by $Q$. 

Let $\chi_V$ be the character of $V$ and $T:=\exp(i\h_{\Bbb R})$. Let $P$ be the weight lattice of $\g$, $P_+$ its dominant part, and for $\mu\in P_+$ let $L_\mu$ be the irreducible 
representation of $G$ with highest weight $\mu$. Let $W$ be the Weyl group of $G$, 
$\rho:=\frac{1}{2}\sum_{\alpha\in R_+}\alpha$ and 
$\alpha^\vee\in \h$ be the coroot corresponding to $\alpha\in R_+$. 
By the Weyl character formula, if $M$ is a finite dimensional representation of $G$ 
then 
$$
[M:L_\mu]=(\chi_M,\chi_\mu)= \sum_{w\in W} (-1)^w \dim M[
\mu +\rho-w \rho].
$$
So by the Weyl dimension formula we have 
$$
\frac{d_n(V,s)}{(\dim V)^n}=\sum_{\mu\in P_+}\frac{[V^{\otimes n}:L_\mu]}{(\dim V)^n}(\dim L_\mu)^s=
$$
$$
\sum_{\mu\in P_+,w\in W}(-1)^w\left(\prod_{\alpha\in R_+}\frac{(\mu+\rho,\alpha^\vee)}{(\rho,\alpha^\vee)}\right)^sP\left(\eta_n=\frac{\mu+\rho-w\rho}{\sqrt{n}}\right).
$$
So Theorem \ref{CLT} yields \scriptsize
 $$
\frac{d_n(V,s)}{(\dim V)^n}\sim \frac{(\det Q)^{-\frac{1}{2}}}{(2\pi n)^{\frac{r}{2}}}\sum_{\mu\in \rho+P_+} (-1)^w\left(\prod_{\alpha\in R_+}\frac{(\mu,\alpha^\vee)}{(\rho,\alpha^\vee)}\right)^s\exp\left(-\frac{Q^{-1}(\mu-w\rho)}{2n}\right)\sim
$$
$$
\frac{(\det Q)^{-\frac{1}{2}}}{(2\pi n)^{\frac{r}{2}}}\sum_{\mu\in \rho+P_+,w\in W} (-1)^w\left(\prod_{\alpha\in R_+}\frac{(\mu,\alpha^\vee)}{(\rho,\alpha^\vee)}\right)^s\exp\left(-\frac{Q^{-1}(\mu)}{2n}+\frac{(B_Q(\mu),w\rho)}{n}\right)=
$$
$$
\frac{(\det Q)^{-\frac{1}{2}}}{(2\pi n)^{\frac{r}{2}}}\sum_{\mu\in \rho+P_+} \left(\prod_{\alpha\in R_+}\frac{(\mu,\alpha^\vee)}{(\rho,\alpha^\vee)}\right)^se^{-\frac{Q^{-1}(\mu)}{2n}}\prod_{\alpha\in R_+}\left(e^{\tfrac{(B_Q(\mu),\alpha)}{2n}}-
e^{-\tfrac{(B_Q(\mu),\alpha)}{2n}}\right).
$$
\normalsize where $B_Q:\h^*\to \h$ is the linear operator corresponding to $Q^{-1}$. 

On the right hand side we have a Riemann sum. Thus we obtain \eqref{asy} with  
$$
C_V(s)=\frac{(\det Q)^{-\frac{1}{2}}}{(2\pi)^{\frac{r}{2}}|W|}\prod_{\alpha\in R_+}(\rho,\alpha^\vee)^{-s}\int_{\h_{\Bbb R}^*}e^{-\frac{Q^{-1}(\mu)}{2}}\prod_{\alpha\in R_+}
|(\mu,\alpha^\vee)|^{s}|(B_Q(\mu),\alpha)| d\mu,
$$
where $d\mu$ is the Lebesgue measure for which ${\rm Vol}(\h_{\Bbb R}^*/P)=1$.  
This proves Theorem \ref{th4}. 
\end{proof} 

\subsection{Computation of $C_V(s)$} 
To compute $C_V(s)$, note that we have a decomposition $\g=\oplus_k \g_k\oplus \mathfrak{z}$, where $\mathfrak{z}$ is the center and $\g_k$ are simple, and $\h=\oplus_k \h_k\oplus \mathfrak{z}$, where 
$\h_k\subset \g_k$ are Cartan subalgebras. Also note that
$$
C_V(s)=\prod_k C_{V|_{\g_k}}(s).
$$ 
Thus it suffices to compute $C_V(s)$ when $G$ is simple. 

In this case 
we have $Q^{-1}(\mu)=\gamma_V(\mu,\mu)$, where $(\cdot,\cdot)$ is the inner product 
for which long roots have squared length $2$. Thus we get 
$$
C_V(s)=\frac{\gamma_V^{\frac{(1-s)|R_+|}{2}}}{(2\pi)^{\frac{r}{2}}|W|\ell^{|R_+^{\rm sh}|}}\prod_{\alpha\in R_+}(\rho,\alpha^\vee)^{-s}\int_{\h_{\Bbb R}^*}e^{-\frac{(\mu,\mu)}{2}}\prod_{\alpha\in R_+}
|(\mu,\alpha^\vee)|^{s+1}d\mu,
$$
where $\ell$ is the lacedness of $G$, and $R_+^{\rm sh}$ is the set of short positive roots. 

Now we will use the Macdonald-Mehta-Opdam identity, \cite{Op}:
$$
(2\pi)^{-\frac{r}{2}}\int_{\h^*_{\Bbb R}}e^{-\frac{(\mu,\mu)}{2}}\prod_{\alpha\in R_+}
|(\mu,\alpha^\vee)|^{s+1}d\mu=\ell^{\frac{s+1}{2}|R_+^{\rm sh}|}\prod_{j=1}^r\frac{\Gamma(1+\frac{s+1}{2}d_j)}{\Gamma(\frac{s+3}{2})}.
$$
Note also that $\prod_{\alpha\in R_+}(\rho,\alpha^\vee)=\prod_{j=1}^r (d_j-1)!$, where 
$d_j$ are the degrees of basic $W$-invariants.

Thus we get

\begin{proposition} 
$$
C_V(s)=\frac{\gamma_V^{\frac{1-s}{2}|R_+|}}{|W|\ell^{\frac{1-s}{2}|R_+^{\rm sh}|}} \prod_{j=1}^r(d_j-1)!^{-s}\frac{\Gamma(1+\frac{s+1}{2}d_j)}{\Gamma(\frac{s+3}{2})}.
$$
\end{proposition} 

Note that in the trivial case $s=1$ where $C_V(s)=1$ this reduces to the well known identity 
$|W|=\prod_{j=1}^r d_j$. Also for $s=0$ we have 
$$
C_V(0)=\frac{\gamma_V^{\frac{1}{2}|R_+|}}{|W|\ell^{\frac{1}{2}|R_+^{\rm sh}|}} \prod_{j=1}^r\frac{\Gamma(1+\frac{d_j}{2})}{\Gamma(\frac{3}{2})}=
\frac{\gamma_V^{\frac{1}{2}|R_+|}}{\pi^{\frac{r}{2}}\ell^{\frac{1}{2}|R_+^{\rm sh}|}} \prod_{j=1}^r\Gamma(\tfrac{d_j}{2}).
$$

\begin{example} Let $G=SL_2(\Bbb C)$ and $V=\oplus_j m_jL(j)$ 
where $L(j)$ has highest weight $j$. Then if $h=\diag(1,-1)$ then 
$$
Q(h)=\frac{{\rm Tr}_V(h^2)}{\dim V}=\frac{\sum_j\frac{j(j+1)(j+2)m_j}{6}}{\sum_j (j+1)m_j}.
$$
Thus for $\mu\in \h^*\cong \Bbb C$
$$
Q^{-1}(\mu)=\frac{\sum_j(j+1)m_j}{\sum_j \frac{j(j+1)(j+2)}{3}m_j}\mu^2,
$$
so 
$$
\gamma_V=6\frac{\sum_j (j+1)m_j}{\sum_j j(j+1)(j+2)m_j}
$$
and, using the doubling formula for the $\Gamma$-function, 
$$
C_V(s)=\frac{{\gamma_V}^{\frac{1-s}{2}}}{2}\frac{\Gamma(s+2)}{\Gamma(\frac{s+3}{2})}=
\pi^{-\frac{1}{2}}{\gamma_V}^{\frac{1-s}{2}}2^{s}\Gamma(\tfrac{s}{2}+1).
$$
In particular, if $V=L(j)$, we get 
$$
C_{L(j)}(s)=2\pi^{-\frac{1}{2}}\left(\frac{2j(j+2)}{3}\right)^{\frac{s-1}{2}}\Gamma( \tfrac{s}{2}
+1).
$$
For example, 
$$
C_{L(j)}(0)=\frac{1}{\sqrt \pi}\left(\frac{j(j+2)}{6}\right)^{-\frac{1}{2}}
 $$
For $j=1$ this recovers the value $C_{L(1)}(0)=\frac{1}{\sqrt{\pi/2}}$ from \eqref{coin}.

\end{example} 

\section{Positive characteristic} 

\subsection{A lower bound for $c_n(V)$}\label{cnv} Let $\C$ be a symmetric tensor category of moderate growth over a field $\bold k$ of characteristic $p>0$. Let $V\in \C$ and $K_V$ be the maximal value of $\delta(W)$ over all 
indecomposables $W$ that occur as direct summands in $V^{\otimes n}\otimes V^{*m}$ for $n,m\ge 0$ (it exists by a straightforward generalization of Corollary 8.17 of [CEO]). It is clear that 
\begin{equation}\label{prodfo} 
K_{V_1\oplus V_2}\le K_{V_1}K_{V_2}.
\end{equation}  

\begin{proposition}\label{p0} For all $n$ we have $c_n(V)\ge K_V^{-1}$. Thus $c(V)\ge K_V^{-1}$. 
\end{proposition} 

\begin{proof} This is immediate from the result of \cite{CEO} that $\delta$ is a ring homomorphism. Indeed, we have 
$$
\delta(V)^n=\delta(V^{\otimes n})=\sum_W [V^{\otimes n}: W]\delta(W)\le 
K_Vd_n(V),
$$
so $c_n(V)=\frac{d_n(V)}{\delta(V)^n}\ge K_V^{-1}$. 
\end{proof} 

\subsection{An upper bound on dimensions of simple objects in a semisimple symmetric tensor category}
As we have explained, the Jordan-Schur inequality \eqref{jsb} in particular 
gives a uniform upper bound on dimensions of irreducible representations of a finite subgroup of $GL(d,\Bbb C)$ in terms of $d$. In other words, it is an upper bound for the Frobenius-Perron dimension of a simple object in the category $\Rep(G)$ generated by $V=\Bbb C^d$. For general fusion categories such a bound does not exist: e.g., the braided Verlinde category ${\rm Ver}_q(SL_2)$ over $\Bbb C$
with $q=e^{\pi i/p}$ is generated by a simple object $V$ such that ${\rm FPdim}V=q+q^{-1}<2$, but has simple objects of ${\rm FPdim}\sim \frac{1}{\sin\frac{\pi}{p}}$, which tends to $\infty$ as $p\to \infty$. In fact, this computation shows that even for symmetric categories there is no such uniform bound if we allow arbitrary positive characteristic (namely, ${\rm Ver}_p$ is a counterexample). However, there is a bound dependent on $p$, and in fact it extends to an arbitrary semisimple symmetric tensor category. Namely, we have the following theorem. 

\begin{theorem}\label{th5} There exists a function $f_p: [1,\infty)\to \Bbb R_+$ such that for every semisimple symmetric tensor category 
$\C$ in characteristic $p$ generated by an object $V$ with $\delta(V)=d$ we have $K_V\le f_p(d)$. 
Namely, we may take 
\begin{equation} 
f_p(d)=e^{a_pd}, 
\end{equation}
for $p=2,3$ and 
\begin{equation} 
f_p(d)=e^{a_pd}2^{\frac{\pi}{2}(p-2)d^2}
\end{equation}
for $p\ge 5$, for a suitable $a_p>0$. 
Specifically, we may take $a_2= \frac{4}{3}\log 3=1.464..$, $a_3= \frac{\log 3}{3}+\log 7=2.312...$  
\end{theorem} 

\begin{remark} Note that this does not hold for non-semisimple (even finite, Frobenius exact) symmetric tensor categories; a counterexample is the sequence of categories $\Rep_{\bold k}(SL_2(\Bbb F_{p^n}))$ for a fixed $p={\rm char}(\bold k)$. 
\end{remark}

Theorem \ref{th5} is proved in the next subsection. Note that by virtue 
of \eqref{prodfo} it suffices to prove it for simple $V$. 

\subsection{Theorem \ref{th5} implies Theorem \ref{th1}}
Let $\C,V$ be as in Theorem \ref{th1}. 
Consider the tensor subcategory $\overline \C_V=\langle \overline V\rangle \subset \overline\C$ generated inside the semisimplification $\overline \C$ 
by the image $\overline V$ of the object $V$. Applying Theorem \ref{th5} 
to $\overline \C_V$ and using Proposition \ref{p0}, we obtain Theorem \ref{th1}. 

\subsection{Proof of Theorem \ref{th5} for $p=2$}  

\begin{proposition}\label{p1} If $p=2$ then for a simple $V$ we have 
$$
K_V\le 3^{\frac{4}{3}\delta(V)-1}.
$$ 
\end{proposition} 

\begin{proof} 
Recall from \cite{CEO} that $\C$ has the form 
${\rm Rep}(\Gamma\cdot (T\times A^\vee))$, where $T$ is a torus, $A$ a finite abelian 2-group, and $\Gamma$ a group of odd order (here the dot denotes an abelian extension, possibly nontrivial). The simple object $V$ then gives rise to a $\Gamma$-orbit $O$ on $X^*(T)\times A$ where $X^*(T)$ is the character lattice of $T$. It is clear that $O$ spans $X^*(T)\times A$ and $|O|\le \delta(V)=\dim_{\bold k}V$. 

Let $\widetilde\Gamma\subset \Gamma\cdot T$ be a lift of $\Gamma$, finite of odd order. 
It acts faithfully on the space $V$ of dimension $\delta(V)$. 
Hence by \cite{Ro1},\footnote{We note that weaker bounds for this index of the form $C^{\delta(V)}$ with $C>3$ were obtained earlier 
by Dornhoff and Low, see \cite{Ro1,Ro2}.} $\widetilde\Gamma$ has a normal abelian subgroup 
$N$ of index $\le 3^{\delta(V)-1}$.

Now consider the orbits of $N$ on $O$. Since $N$ is normal, these orbits $O_1,...,O_k$ 
have the same size $r:=|O|/k$. Let $N_i\subset N$ be the stabilizer of the orbit $O_i$, so $[N:N_i]=r$. 
Let $M=\cap_i N_i$, so 
$$
[N:M]\le r^k=r^{|O|/r}.
$$ 
Note that by construction $M$ commutes with $T\times A^\vee$, as it acts trivially on $O$ which spans $X^*(T)\times A$. Define the group $Q:=M\cdot (T\times A^\vee)$. Then $Q$ is a normal abelian subgroup scheme 
of $\Gamma\cdot (T\times A^\vee)$, which thus has an eigenvector in every irreducible 
representation. It follows that 
$$
K_V\le [\Gamma \cdot (T\times A^\vee):Q]\le 3^{\delta(V)-1}r^{\delta(V)/r}\le 3^{\frac{4}{3}\delta(V)-1}. 
$$
(the maximum is attained for $r=3$). 

\subsection{Proof of Theorem \ref{th5} for $p>2$}
Consider the fiber functor $F: \C\to {\rm Ver}_p$, which exists and is unique by \cite{CEO}. Let $G={\underline {\rm Aut}}_{\otimes}F$, $\pi:=\pi_1({\rm Ver}_p)$ (linearly reductive affine group schemes in ${\rm Ver}_p$), and $\phi: \pi\to G$ be the canonical homomorphism. Then $\C\cong \Rep(G,\phi)$, the category of representations of $G$ on which the action of $\pi$ via $\phi$ is the canonical action.  

Consider first the case when $G$ is connected. 
Since $G$ is linearly reductive, we have 
$G=\exp(\g_-)\cdot (A^\vee \times T)$, where $A$ is a finite abelian $p$-group, $A^\vee$ is the dual infinitesimal group scheme, $T$ is a torus, and $\g_-$ is the part 
of $\g:={\rm Lie}G$ comprised by nontrivial objects of ${\rm Ver}_p$. 
Namely, this follows from the fact that every closed subgroup scheme of a linearly reductive group scheme of finite type (containing the image of $\phi$) is linearly reductive, which for finite group schemes is proved in the same way as Lemma 8.2 of 
\cite{EOV} 
and then follows in general by Theorem 8.11 of 
\cite{CEO}. 

Every simple object $W$ of $\C$ is generated by a simple subobject $L_k$ of ${\rm Ver}_p$ on which $A^\vee\times T$ acts by a character. 
Thus 
$$
\delta(W)\le {\rm FPdim}S\g_- {\rm FPdim}L_k
\le \frac{{\rm FPdim}S\g_-}{\sin \frac{\pi}{p}}.
$$ 
So it remains to estimate ${\rm FPdim}S\g_-$. 

If $p=2,3$ then $\g_-=0$ (see \cite{CEO}, Section 8), so we may assume that $p\ge 5$. Note that $\g_-\subset \mathfrak{gl}(F(V))=F(V)\otimes F(V)^*$. 
Thus, writing $F(V)=\oplus_i m_iL_i$, we have
$$
\dim \Hom_{{\rm Ver}_p}(L_r,\g_-)\le \sum_{i,j} C_{ij}^rm_im_j,\ r\ge 2,
$$
where $C_{ij}^r$ are the structure constants of the Verlinde algebra (the Grothendieck ring of ${\rm Ver}_p$). 
Since ${\rm FPdim}SL_r={\rm FPdim}\wedge L_{p-r}\le 2^{p-r}$ for $r\ge 2$, this means that the Frobenius-Perron dimension of $S\g_-$ satisfies the inequality 
$$
{\rm FPdim}S\g_-\le \prod_{r\ge 2} ({\rm FPdim}\wedge L_{p-r})^{\sum_{i,j} C_{ij}^rm_im_j}
$$
$$
\le 2^{\sum_{r\ge 2,i,j}C_{ij}^{r}(p-r)m_im_j}\le  2^{(p-2)\sum_{i,j,r}C_{ij}^rm_im_j}
$$
$$
\le 
2^{(p-2)\sum_{i,j}\min(i,p-i,j,p-j)m_im_j}.
$$
As 
$$
\min(k,p-k)\le \tfrac{p}{2}[k]_q\sin \tfrac{\pi}{p}\le \tfrac{\pi}{2}[k]_q,
$$
we have 
$$
{\rm FPdim}S\g_-\le 2^{\frac{\pi}{2}(p-2)d^2}.
$$
Thus
$$
\delta(W)\le \frac{2^{\frac{\pi}{2}(p-2)d^2}}{\sin \tfrac{\pi}{p}}.
$$

Now consider the general case ($G$ is not necessarily connected).  It turns out that 
Robinson's exponential Jordan-Schur bound for odd order groups \cite{Ro1}  extends (with a similar proof) to groups of order coprime to $p$ with $p>2$. Namely, introduce the following function on the set of odd primes: 
$$
C(3)=\frac{7}{3},\ C(p)=\tfrac{1}{p}\max(60,(p-1)!^{\frac{1}{p-3}}),\ p\ge 5.
$$
We note that $C$ is decreasing starting from $p=5$ (so $C(p)\le C(5)=12$), $\lim_{p\to \infty}C(p)=\frac{1}{e}$, and the second entry of the maximum wins iff $p>152$.

\begin{theorem}\label{ronew} (\cite{Ro2}) Let $p$ be an odd prime. 
If $G\subset GL(d,\bold k)$ is a group of order prime to $p$ then 
$G$ has an abelian normal subgroup $N$ with 
\footnote{There is a simple example in 
\cite{Ro2} 
showing that one cannot do better.}  
\begin{equation}\label{JS}
|G/N|\le (C(p)p)^{d-1}.
\end{equation} 
\end{theorem} 

From this, arguing
similarly to the case $p=2$ considered above, we get 
\begin{equation} 
\delta(W)\le (3^{1/3}\cdot 7)^{d-1}
\end{equation} 
for $p=3$ and 
\begin{equation} 
\delta(W)\le \frac{1}{\sin \frac{\pi}{p}}(3^{1/3}C(p)p)^{d-1}2^{\frac{\pi}{2}(p-2)d^2}
\end{equation}
for $p\ge 5$. 
This completes the proof. 
\end{proof}

\subsection{Conjecture \ref{eveneasier} implies Conjecture \ref{co1}}
Conjecture \ref{eveneasier} implies that if $G_1\subset GL(\oplus_i n_iL_i)$ is linearly reductive of height $1$, $\g={\rm Lie}G_1$ and $d=\sum n_i[i]_q$ then ${\rm FPdim}\g\le M_pd$. Thus if $\g=\oplus_{i=1}^{p-1} m_iL_i$ then $m_i\le M_pd$ for all $i$. 
Hence 
$$
{\rm FPdim} S\g_-=\prod_{i=2}^{p-1} ({\rm FPdim}S(L_i))^{m_i}=\prod_{i=2}^{p-1} ({\rm FPdim}\wedge L_{p-i})^{m_i} 
$$
$$
\le
2^{\sum_{i=2}^{p-1}m_i(p-i)}\le 2^{\frac{M_p(p-1)(p-2)}{2}d}.
$$
Now let $G=\underline{\rm Aut}F\subset GL(F(V))$ as in the proof of Theorem \ref{th5}, so that $G=\Gamma\cdot \exp(\g_-)(A^\vee \times T)$, where $\Gamma$ is a finite group of order coprime to $p$. Also let $G_1\subset G$ be the Frobenius kernel, so $G_1$ is linearly reductive of height $1$ and ${\rm Lie}G={\rm Lie}G_1=\g$. Then every simple $G$-module $W$ is generated by an object $L_k$ on which $A^\vee\times T$ acts by a character. So arguing as in the proof of Theorem \ref{th5} and using Theorem \ref{ronew}, we have 
$$
\delta(W)\le\frac{1}{\sin \frac{\pi}{p}}(3^{\frac{1}{3}}C(p)p)^{d-1} \cdot 2^{\frac{M_p(p-1)(p-2)}{2}d} \le e^{A_pd}
$$
for a suitable $A_p>0$. Thus Conjecture \ref{eveneasier} 
implies Conjecture \ref{co1}. 

\begin{example} It will be shown in \cite{EKO} that for $p=5$ the simple Lie algebra $\g$ must be isomorphic to $L_3$ as an object of ${\rm Ver}_p$. So $D_p=\frac{1+\sqrt{5}}{2}$, thus we may take $M_p=(\frac{1+\sqrt{5}}{2})^2$. So for $p=5$ we have 
$\lambda(V)<16\delta(V)$. 
\end{example} 

\section{A conjectural classification of semisimple symmetric tensor categories in positive characteristic} 

\subsection{The conjecture} 
The goal of this section is to present a conjectural classification of semisimple symmetric tensor categories of moderate growth over an algebraically closed field $\bold k$ of characteristic $p>0$. 

Let $\C$ be a semisimple symmetric tensor category over $\bold k$ of moderate growth. Recall that by \cite{CEO}, $\C$ admits a fiber functor 
$F: \C\to {\rm Ver}_p$, hence $\C\cong {\rm Rep}(G,\phi)$, where 
$G$ is a linearly reductive affine group scheme 
 in ${\rm Ver}_p$. Thus it remains to classify such group schemes. 

In characteristic zero (setting ${\rm Ver}_0={\rm SVec}_{\bold k}$, the category of super-vector spaces) such group schemes are exactly the (pro-)reductive supergroups (where we include ${\rm SOSp}(1|2n)$ as possible factors).   
But non-abelian connected reductive supergroups are not linearly reductive in characteristic $p$. However, they have linearly reductive avatars which we call {\it Verlinde categories of reductive groups} (\cite{GK,GM}). Namely, let $Q$ be a connected Dynkin diagram, $h_Q$ be the Coxeter number of $Q$ and $p>h_Q$. Let ${\rm Tilt}_p(G_Q)$ be the category of tilting modules 
over the simply connected group $G_Q(\bold k)$. Then we define the Verlinde category ${\rm Ver}_p(G_Q)$ to be the semisimplification of ${\rm Tilt}_p(G_Q)$, which is a symmetric fusion category (see the next section for a brief discussion of semisimplification of non-abelian monoidal categories). For instance, if $Q=A_1$ then
$G_Q=SL(2)$ and ${\rm Ver}_p(SL(2))$ is the usual Verlinde category ${\rm Ver}_p$. One can show that ${\rm Ver}_p(G_Q)={\rm Ver}_p^+(G_Q)\boxtimes 
\Rep(Z_Q,z_Q)$, where $Z_Q\subset G_Q$ is the center and $z_Q\in Z_Q$ is the image of $-1$ under the principal map $SL(2)\to G_Q$.  

Another new feature in characteristic $p>0$ is that the finite group schemes $(\Bbb Z/p^n)^\vee$, whose representation categories ${\rm Rep}(\Bbb Z/p^n)^\vee=\Vec_{\Bbb Z/p^n}$ are pointed semisimple symmetric tensor categories, become connected (even infinitesimal).  

Our main conjecture is as follows. 

\begin{conjecture}\label{mai} Let $\C$ be a finitely generated semisimple symmetric tensor category of moderate growth over $\bold k$. Then $\C$ can be obtained from a finite Deligne product 
of categories $\Rep(\Bbb G_m)$, $\Vec_{\Bbb Z/p^n}$ and ${\rm Ver}_p^+(G_Q)$ by the following operations, in the given order: 

(1) changing the commutativity isomorphism using a $\Bbb Z/2$-grading 
on the category;

(2) equivariantization under a finite group of order coprime to $p$. 

In other words, every such category is the equivariantization 
of a Deligne tensor product of a pointed category (without $p'$-torsion) with a product of ${\rm Ver}_p^+(G_Q)$ by a finite group of order coprime to $p$.  
\end{conjecture} 

It is clear that Conjecture \ref{mai} implies a similar classification of symmetric tensor categories of moderate growth over $\bold k$ 
which are not necessarily finitely generated: we just need to allow infinite Deligne products and 
allow profinite groups in (2). 

\begin{remark} 1. The categories ${\rm Ver}_p^+(G_Q)$ are defined for all $p$, not just $p>h_Q$, but we expect that for smaller $p$ they express as 
indicated in Conjecture \ref{mai}, so that they are ``redundant". For $Q=A_{n-1}$ this expression is known, see \cite{BEEO}, but for other types with a few exceptions we don't know such an expression. 

2. There are level-rank dualities 
$$
{\rm Ver}_p^+(SL(n))\cong {\rm Ver}_p^+(SL(p-n)),
$$ 
so for type $Q=A_{n-1}$ we may assume that $p>2n$.

3. There are level-rank dualities 
$$
{\rm Ver}_p^+({\rm Spin}(2n+1))\cong 
{\rm Ver}_p^+({\rm Sp}(p-1-2n)),
$$ 
so we may eliminate type $C_n$ from the conjecture.  

4. Similarly to ${\rm Ver}_p(G_Q),{\rm Ver}_p^+(G_Q)$, we can define the categories ${\rm Ver}_p({\rm SOSp}(1|2n))={\rm Ver}_p^+({\rm SOSp}(1|2n))$ for $p>2n-1$ (note that the Coxeter number of ${\rm SOSp}(1|2n)$ is $2n-1$ and that its center is trivial). However, we have level-rank dualities 
$$
{\rm Ver}_p^+({\rm Spin}(2n))\cong {\rm Ver}_p^+({\rm SOSp}(1|p+1-2n)),
$$ 
which is why these categories are not needed in Conjecture \ref{mai}. 

For example, the simplest of these categories, ${\rm Ver}_p({\rm SOSp}(1|2))$ is equivalent to the category  ${\rm Ver}_p^+({\rm Spin}(p-1))$. Thus ${\rm Ver}_p({\rm OSp}(1|2))\cong {\rm Ver}_p^+({\rm Spin}(p-1))^{\Bbb Z/2}$, where 
$\Bbb Z/2$ swaps the horns of the Dynkin diagram $D_{\frac{p-1}{2}}$. 
This category arises in the study of the semisimplification of 
the modular representation category of $\Bbb Z/p^n$, $n>1$, and 
is therefore discussed in detail in Section 5, 
where it is denoted by $\mathcal D$. 
 For instance, for $p=5$ we get 
$$
{\rm Ver}_p({\rm OSp}(1|2))\cong ({\rm Ver}_5^+\boxtimes {\rm Ver}_5^+)^{\Bbb Z/2}
$$
where $\Bbb Z/2$ swaps the factors.\footnote{$SOSp(1|2)$ is a supergroup scheme in ${\rm Ver}_p$, which by definition means that it is an ordinary group scheme in ${\rm Ver}_p\boxtimes {\rm sVec}$, which is equivalent to ${\rm Ver}_p\boxtimes {\rm Vec}_{\Bbb Z/2}$. In the latter realization this group scheme lies in ${\rm Ver}_p^+\boxtimes {\rm Vec}_{\Bbb Z/2}$, so the Verlinde fiber functor maps it to ${\rm Ver}_p^+$. In ${\rm Ver}_p^+$ its Lie algebra has the form $L_3\oplus L_{p-2}$, which is $L_3\oplus L_3$ if $p=5$.}

5. After these reductions, there are still some redundancies 
and ``exceptional isomorphisms" left, so the list of categories needed in Conjecture \ref{mai} is not minimal even if we exclude type $C$. 
For example, if $p=h_Q+1$ is a prime then ${\rm Ver}_p^+(G_Q)=\Vec$, and ${\rm Ver}_{11}(G_2)={\rm Ver}_{11}^+(SL(2))$. We will discuss these exceptional isomorphisms in more detail below. 

6. Let $p\ge 5$. Recall that the group scheme $\pi:=\pi_1({\rm Ver}_p)$ 
has a decomposition $\pi=\pi_+\times \Bbb Z/2$, where 
$\pi_+=\pi_1({\rm Ver}_p^+)$ is connected. 
Let $\s:={\rm Lie}\pi={\rm Lie}\pi_+$. We have 
$\s=L_3$ with the bracket given by 
the unique non-zero morphism $L_3\otimes L_3\to L_3$ 
up to scaling, and $\pi_+=\exp(\s)$ (indeed, it is easy to check that 
$U(\s)^*\cong \oplus_{j=1}^{\frac{p-1}{2}}L_{2j-1}\otimes L_{2j-1}^*$ 
as Hopf algebras). 

Let $\g_{Q,p}$ be the image of $\g_Q:={\rm Lie}G_Q$ in ${\rm Ver}_p^+$ (viewed as a module over $\bold k[x]/(x^p)$ via commutator with the regular nilpotent). Then $\g_{Q,p}\in {\rm Ver}_p^+$, $\Hom_{{\rm Ver}_p}(\bold 1,\g_{Q,p})=0$, $\g_{Q,p}$ contains a unique copy of $\s$, and 
$G_{Q,p}:=\exp(\g_{Q,p})$ is a simple linearly reductive 
group scheme (of height $1$). Moreover, we
have ${\rm Ver}_p^+(G_Q)=\Rep(G_{Q,p},\pi)$. Thus Conjecture \ref{mai} implies that every linearly reductive group scheme $G$ of height $1$ in ${\rm Ver}_p$ in fact belongs to ${\rm Ver}_p^+$ and is a direct product of copies of 
$(\Bbb Z/p)^\vee$ and $G_{Q,p}$. \end{remark} 

Thus, for $p=2$ and $p=3$, we don't need the categories ${\rm Ver}_p(G_Q)$ and Conjecture \ref{mai} simply says that $\C$ is an equivariantization of a pointed semisimple category with respect to a finite group of order coprime to $p$. These cases of Conjecture \ref{mai} are proved in \cite{CEO}, Section 8; namely, they follow from the main result of \cite{CEO} and Nagata's classification 
of semisimple affine group schemes. 

For $p=5$, Conjecture \ref{mai} says that $\C$ is the equivariantization  
of the product of a pointed semisimple category with $({\rm Ver}_p^+)^{\boxtimes n}$ under a finite group of order coprime to $p$. 
This will be proved in \cite{EKO}. 

For every $p\ge 7$ the conjecture remains open. For example, for $p=7$ 
Conjecture \ref{mai} says that $\C$ is the equivariantization  
of the product of a pointed semisimple category with $({\rm Ver}_p^+)^{\boxtimes n}\boxtimes {\rm Ver}_p^+(SL(3))^{\boxtimes m}$ under a finite group of order coprime to $p$.

\subsection{Examples} \label{examp}
Let us now list the categories ${\rm Ver}_p^+(G_Q)$
that need to be included for each $p$ together with the Lie algebra 
$\g=\g_{Q,p}:={\rm Lie}G_{Q,p}$ decomposed in ${\rm Ver}_p^+$ (where $L_n$ is shortly denoted by $n$). 

The decomposition of $\g$ is determined as follows.
We write the sum $\sum_{i=1}^r (2m_i+1)$, where $m_i$ are the 
exponents of $Q$, and then remove $p$ (if present) and all 
the numbers $k>p$ together with the ``mirror" copy of $2p-k$ (one can check that 
if $k$ is present, so is $2p-k$). Not surprisingly, this agrees with the Jordan decomposition of ${\rm ad}(e)$ where $e$ is the regular nilpotent element 
of ${\rm Lie}G_{Q}$ in characteristic $p$ (the blocks of size $<p$), see \cite{St}. 

Concretely, for each $p\ge 5$ we have the following series of categories: 

1. Q=$A_{m-1}$, $1\le m\le \frac{p-1}{2}$, 
$$
\g=3+5+...+(2m-1).
$$  

2. $p\ge 11$, $Q=B_m$, $2\le m\le \frac{p-7}{2}$, 
$$
\g=3+7+...+(4\min(m,\tfrac{p-1}{2}-m)-1).
$$ 
Note that due to level-rank duality, $B_m$ gives the same category as $C_{\frac{p-1}{2}-m}$. 

3. $p\ge 11$, $Q=D_m$, $4\le m\le \frac{p-1}{2}$. Then pattern for $\g$ is as follows: for 
$Q=D_{\frac{p-1}{2}-r}$ with $0\le r\le \frac{p-7}{4}$ we have 
$$
\g=[3+7+...+(4r+3)]+(p-2-2r),
$$
and the rest of the answers are the same as in characteristic zero. 

We now list the decompositions of $\g$ in exceptional types, omiting the cases when the decomposition is the same as in characteristic zero. We also point out exceptional 
isomorphisms when applicable.  

\vskip .1in

$p=13$: 

We have an exceptional isomorphism 
${\rm Ver}_{13}(G_2)\cong {\rm Ver}_{13}^+(D_6)$. 

\vskip .1in

$p=17$:

$Q=F_4$, $\g=3+15$

$Q=E_6$, $\g=3+9+15$

We have an exceptional isomorphism 
${\rm Ver}_{17}(F_4)\cong {\rm Ver}_{17}^+(D_8)$. 

\vskip .1in

$p=19$: 

$Q=F_4$, $\g=3+11$

$Q=E_6$, $\g=3+9+11+17$

We have an exceptional isomorphism 
${\rm Ver}_{19}(G_2)\cong {\rm Ver}_{19}(F_4)$. 

\vskip .1in

$p=23$: 

$Q=F_4$, $\g=3+11+15$

$Q=E_6$, $\g=3+9+11+15+17$

$Q=E_7$, $\g=3+15$

\vskip .1in

$p=29$: 

$Q=E_7$, $\g=3+11+15+19+27$

\vskip .1in

$p=31$: 

$Q=E_7$, $\g=3+11+15+19+23$

\vskip .1in

$p=37$:

$Q=E_8$, $\g=3+23$ 

\vskip .1in

$p=41$:

$Q=E_8$, $\g=3+15+27+39$

\vskip .1in 

$p=43$:

$Q=E_8$, $\g=3+15+23+35$

\vskip .1in 

$p=47$:

$Q=E_8$, $\g=3+15+23+27+39$

\vskip .1in 

$p=53$:

$Q=E_8$, $\g=3+15+23+27+35+39$

\vskip .1in 

$p=59$:

$Q=E_8$, $\g=3+15+23+27+35+39+47.$ 
 
\begin{remark} All simple Lie algebras $\g_{Q,p}$ have a non-degenerate Killing form. 
This follows from the following self-dual version of the Freudenthal - de Vries strange formula (see 
\cite[p.153]{Su}):\footnote{This formula reduces to the usual Freudenthal - de Vries formula $(\rho,\rho)=\frac{h_Q\dim \g_Q}{12}$ in the simply-laced case and is easy to check by hand.} if $\g_Q$ is a simple complex Lie algebra of rank $r$ and $m_i$ are the exponents of $Q$ then 
$$
\sum_{i=1}^r \frac{m_i(m_i+1)(2m_i+1)}{3}=\frac{\ell h^\vee_Q h^\vee_{Q^\vee}\dim \g_Q}{6}=
\frac{\ell h^\vee_Q h^\vee_{Q^\vee}(h_Q+1)r}{6},
$$
where $h_Q$ is the Coxeter number of $Q$, $h^\vee_Q$ is the dual Coxeter number of $Q$ and $\ell$ the lacedness of $Q$. Namely, the left hand side is the ratio of the Killing form of $\g$ when restricted to $\s$ to the Killing form of $\s$.  
\end{remark} 

\subsection{Rank $2$ simple Lie algebras} 

Let $\g$ be a simple Lie algebra in ${\rm Ver}_p^+$ containing 
$\s$, such that $\s$ acts on $\g$ canonically. 
Define the {\it rank} of $\g$ to be the length of $\g$ as an $\s$-module (equivalently, as object of ${\rm Ver}_p^+$). For example, the only simple Lie algebra of rank $1$ 
is $\s=\g_{A_1,p}$. We also observe the following simple Lie algebras $\g=\g_{Q,p}$ of rank $2$ occurring in Subsection \ref{examp}: 

{\bf Series:}

Type $Q=A_2$: $p\ge 7$, $\g=3+5$; 

Type $Q=B_2=C_2$: $p\ge 11$, $\g=3+7$;

Type $D_2^*$ obtained from $Q=D_{\frac{p-1}{2}}$: $p\ge 11$, 
$\g=3+(p-2)$. 

Type $Q=G_2$, $p\ge 17$, $\g=3+11$. 

{\bf Exceptional cases:} 

Type $E_2^*$ obtained from $Q=E_7$, $p=23$, $\g=3+15$. 

Type $E_2^{**}$ obtained from $Q=E_8$, $p=37$, $\g=3+23$. 

\begin{conjecture}\label{nooth} There are no other simple Lie algebras of rank $2$ in ${\rm Ver}_p^+$. 
\end{conjecture} 

Let us give a construction of the above rank $2$ Lie algebras without using 
semisimplification of ordinary Lie algebras. For simplicity we assume that the Lie algebra has an invariant inner product. 

Let $\g=L_3\oplus L_r$ for $r\ge 5$ odd. If $r=p-2$, consider the Lie superalgebra ${\rm osp}(1|2)=L_3\oplus L_2$, which exists because $S^4L_2=L_5$ does not contain invariants, which trivially implies the Jacobi identity. By parity change of $L_2$, we get $D_2^*$. 

In general, we need to check the Jacobi identity in $\Hom(\bold 1,\wedge^4 L_r)$. 
Recall that for large $p$, we have $d_r:=\dim \Hom(\bold 1,\wedge^4 L_r)=a_0(r)-a_2(r)$, where $a_0(r),a_2(r)$ are the coefficients of $1,q^2$ in the Gauss polynomial $\binom{r}{4}_q$. 
This gives $d_5=0$, $d_7=d_9=d_{11}=1,d_{13}=d_{15}=d_{17}=2$, etc. 
This immediately implies the existence of $A_2$, and also $B_2$ and $G_2$ (including characteristic $0$), since in these cases we may choose a scaling of the bracket map $L_r\otimes L_r\to L_r$ to satisfy the Jacobi identity. 

It remains to show the existence of $E_2^*$ and $E_2^{**}$. 
For $E_2^*$, note that for $p=23$, $\dim \Hom(\bold 1, \wedge^4L_{15})=\dim \Hom(\bold 1,S^4L_8)=
a_0(11)-a_2(11)=d_{11}=1$ (so it is less than $d_{15}$ since $p$ is not very big), 
but we can choose the scaling of the bracket map $L_{15}\otimes L_{15}\to L_{15}$, which implies the existence of $E_2^*$.  

The existence of $E_2^{**}$ is more mysterious since for $p=37$, 
$$
\dim \Hom(\bold 1,\wedge^4 L_{23})=\dim \Hom(\bold 1,S^4 L_{14})=a_0(17)-a_2(17)=d_{17}=2,
$$ 
so the 
choice of scaling of $L_{23}\otimes L_{23}\to L_{23}$ does not a priori suffice to ensure the Jacobi identity. However, the determinant of the appropriate 2 by 2 matrix of 6j-symbols
$$
\begin{pmatrix} 11& 11& 11\\ 11 & 11 & a\end{pmatrix} \begin{pmatrix} 11& 11& 1\\ 11 & 11 & b\end{pmatrix} - \begin{pmatrix} 11& 11& 11\\ 11 & 11 & b\end{pmatrix} \begin{pmatrix} 11& 11& 1\\ 11 & 11 & a\end{pmatrix}
$$
(for $a,b$ odd and not equal $1,11$) is miraculously always divisible by 37, so equals zero in the ground field, implying the existence of $E_2^{**}$.  

\begin{remark} Similar calculations provide computational evidence for Conjecture \ref{nooth}.
\end{remark} 

\subsection{Invariantless simple Lie algebras in ${\rm Ver}_p$}

Let $\g$ be an {\it invariantless} simple Lie algebra in ${\rm Ver}_p$, i.e. one with $\Hom_{{\rm Ver}_p}(\bold 1,\g)=0$. In this case, for $p>0$, $G:=\exp(\g)$ is a finite 
simple group scheme in ${\rm Ver}_p$ of height $1$ (although not necessarily linearly reductive). 

For $p=0$, if we replace ${\rm Ver}_p$ by $\Rep(SL(2))$, such Lie algebras are easy to classify. 
Namely, the radical of $\g$ is $SL(2)$-invariant, so must be zero, i.e. $\g$ is semisimple, hence simple in the usual sense. 
Also the action of $SL(2)$ must be inner, so defined by a homomorphism $\mathfrak{sl}_2\to \g$.
Moreover, the nilpotent element $e\in \g$ corresponding to this homomorphism must be {\it distinguished}, i.e. its centralizer in $\g$ consists of nilpotent elements 
(such non-regular elements exist for types $B_n,n\ge 4$, $C_n,n\ge 3$, $D_n,n\ge 4$, and exceptional types, see \cite{CM}). 
Conversely, every distinguished $e$ gives rise to such a Lie algebra. Thus such Lie algebras correspond to pairs $(\g,e)$ where $\g$ is a simple Lie algebra and $e\in \g$ a distinguished nilpotent element up to conjugation. 

The same method can be used to construct invariantless simple Lie algebras in ${\rm  Ver}_p$ at least for large enough $p$. Namely, 
let $\g_p$ be the standard reduction of $\g$ to characteristic $p$ and 
$e\in \g_p$ a distinguished nilpotent element such that 
${\rm ad}(e)$ has nilpotency order  $\le p$. Then we may view 
$\g_p$ as a Lie algebra in $\bold k[x]/(x^p)$ -mod where $x$ is primitive and acts  
by ${\rm ad}(e)$, and take $\g_{p,e}$ to be the image of $\g_p$ in the semisimplification ${\rm Ver}_p$ of $\bold k[x]/(x^p)$-mod (such constructions are discussed in \cite{K}). 

\begin{question}\label{inva} Are all invariantless simple Lie algebras in ${\rm Ver}_p$ obtained in this way? 
\end{question} 

Note that all the Lie algebras of simple linearly reductive group schemes from Subsection \ref{examp} are invariantless and obtained by this method, where 
$e\in \g$ is the principal nilpotent element. Here are some other examples. 

\begin{example} 1. We can take $e$ to be the subregular element in $\g=\mathfrak{so}(8)$ (Jordan type (5,3)): take the Lie subalgebra $\mathfrak{so}(3)\oplus \mathfrak{so}(5)\subset \mathfrak{so}(8)$ and let $e$ be the principal nilpotent of this subalgebra. Then we
have the decomposition 
$$
\mathfrak{so}(8)=\mathfrak{so}(3)\oplus \mathfrak{so}(5)\oplus V_3\otimes V_5,
$$
where $V_m$ is the vector representation of $\mathfrak{gl}(m)$. 
Thus the decomposition of $\g_{p,e}$ for $p\ge 11$ has the form 
$$
\g_{p,e}=3+(3+7)+(3+5+7)=3+3+3+5+7+7,
$$
while for $p=7$ we have 
\begin{equation}\label{g2e} 
\g_{p,e}=3+3+3+5. 
\end{equation} 

2. Note that Example \ref{g2e} (for all $p\ge 7$) is also obtained 
from $\g=G_2$ and principal nilpotent $e$ of the Lie subalgebra 
$\mathfrak{sl}_2\oplus \mathfrak{sl}_2\subset \g$.

3. Another small example is given by $\g=\mathfrak{sp}(6)$ and $e$ the 
regular nilpotent element of the subalgebra $\mathfrak{sl}(2)\oplus \mathfrak{sp}(4)$ (Jordan type (4,2)). In this case we get 
$$
\g_{e,p}=3+(3+7)+V_2\otimes V_4=3+(3+7)+(3+5)=3+3+3+5+7
$$  
for $p\ge 11$, and for $p=7$ we again recover Example \ref{g2e}. 
\end{example} 

It follows from 
\cite{EKO} that for $p=5$ the only invariantless simple Lie algebra is $\s$, so the answer to Question \ref{inva} is affirmative. 

 \section{Semisimplification of the representation category of a cyclic group}

\subsection{Fusion rules for representations of $\Bbb Z/p^n$.}\label{Grform} Let ${\bold k}$ be a field of characteristic $p>0$ and let  $\Bbb Z/{p^n}$ be a cyclic group of order $p^n$ ($n\in {\Bbb Z}_{>0}$).
Let $v_r$ be the class of the indecomposable module of dimension $r$ in the Green ring (= split Grothendieck ring)
$G_n$ of the category of finite dimensional representations of $\Bbb Z/{p^n}$ over ${\bold k}$; then the classes $v_r,
r\in [1,p^n]$ form a basis of $G_n$ over ${\Bbb Z}$. We will use the  convention $v_0=0$ whenever necessary.

We have a surjective homomorphism $\Bbb Z/{p^{n+1}}\to \Bbb Z/{p^n}$  which induces a ring embedding $G_n\hookrightarrow G_{n+1}$.
Note that under this embedding $v_r\in G_n$ maps to $v_r\in G_{n+1}$ which makes the notation $v_r$ unambiguous.
We will use the following formulas from \cite{Gr} (see also 
\cite{B1}, Chapter 5) which give a description of the extension of based rings
$G_n\subset G_{n+1}$. Let $q:=p^n$ and let $w=v_{q+1}-v_{q-1}$. Then
\begin{equation}\label{so3}
\begin{array}{lrl}
wv_r=v_{r+q}-v_{q-r}&(1\le r\le q)&\mbox{\cite[(2.3a)]{Gr}}\\
wv_r=v_{r+q}+v_{r-q}&(q<r< (p-1)q)&\mbox{\cite[(2.3b)]{Gr}}\\
wv_r=v_{r-q}+2v_{pq}-v_{2pq-r-q}&((p-1)q\le r\le pq)&\mbox{\cite[(2.3c)]{Gr}}\\
\end{array}
\end{equation}

\scriptsize
\begin{equation}\label{bar}
\begin{array}{lrl}
v_{q-1}v_r=v_{q-r}+(r-1)v_{q}&(1\le r< q)&\mbox{\cite[(2.5b)]{Gr}}\\
v_{q-1}v_r=(r_1-1)v_{q(r_0+1)}+v_{q(r_0+1)-r_1}+(q-r_1-1)v_{qr_0}&(q\le r\le pq)&\mbox{\cite[(2.9c)]{Gr}}
\end{array}
\end{equation}
\normalsize where in the last formula we write $r=qr_0+r_1$ with $r_0,r_1\in {\Bbb Z}$, $\quad$ $0\le r_1<q$.

 \subsection{Fusion rules for the semisimplification of the representation category of $\Bbb Z/p^n$.}
The classes of indecomposable negligible objects in $G_{n}$ are $v_r$ with $r$ divisible by $p$. Let
$\overline G_n$ be the quotient of $G_n$ by the ideal spanned by these classes and let $u_r$ denote the image
of $v_r$ in $\overline G_n$.  Thus the basis of $\overline G_n$
is given by $u_r$ where 
$$
r\in [1,p^n]^*:=\{ r\in {\Bbb Z} | 1\le r\le p^n, (r,p)=1\}.
$$ 

We now use the formulas from Subsection \ref{Grform} to describe
the based ring extension $\overline G_n\subset \overline G_{n+1}$. It is convenient to use the following notation:
for $r\in [1,p^{n+1}]^*$  with $r=qr_0+r_1$ where $r_0,r_1\in {\Bbb Z}$, $0< r_1<q$, we write $\overline{r}:=qr_0+(q-r_1)$. 
It is clear that  $\overline{\overline{r}}=r$ and $\overline{r\pm q}=\overline{r} \pm q$ (whenever both $r$ and $r\pm q$ are
in $[1,p^{n+1}]^*$).

We have the following consequences of formulas \eqref{bar} above:

\begin{equation}
u_{q-1}u_r=u_{\overline{r}}\; (1\le r\le pq, (r,p)=1).
\end{equation}
In particular for the image $u_{q+1}-u_{q-1}$ of $w$ in $\overline G_{n+1}$ we have
$$
u_{q-1}(u_{q+1}-u_{q-1})=u_{2q-1}-1
$$ and we deduce from \eqref{so3}
$$
(u_{2q-1}-1)u_r=\left\{ \begin{array}{lr} 
u_{\overline{r+q}}-u_{\overline{q-r}}&(1\le r\le q)\\
u_{\overline{r+q}}+u_{\overline{r-q}}&(q<r< (p-1)q)\\
u_{\overline{r-q}}-u_{\overline{2pq-r-q}}&((p-1)q\le r\le pq)
\end{array} \right.
$$
Equivalently,
\begin{equation}\label{mrules}
u_{2q-1}u_r=\left\{ \begin{array}{lr} 
u_{\overline{r} +q}&(1\le r\le q)\\
u_{\overline{r} -q}+u_r+u_{\overline{r} +q}&(q<r< (p-1)q)\\
u_{\overline{r} -q}&((p-1)q\le r\le pq)
\end{array} \right.
\end{equation}

\begin{example}\label{p even}
Assume that $p=2$. Then \eqref{mrules} says that
$$u_{2q-1}u_r=u_{2q-r} \; (1\le r \le 2q-1).$$
This implies that all basis elements of $\overline G_{n+1}$ are invertible;
moreover the group of invertible basis elements of $\overline G_{n+1}$ is generated by the similar group for $\overline G_n$ and
by the element $u_{2q-1}$ of order $2$. We deduce by induction that as a based ring $\overline G_n$ is isomorphic to
the group algebra of an  elementary abelian group of order $2^{n-1}$. We can choose elements $u_3, u_7, \ldots ,
u_{2^n-1}$ for a basis of this group over the field with 2 elements.
For example we have:
$$u_{99}=u_{127}u_{29}=u_{127}u_{31}u_3,$$
$$u_{53}=u_{63}u_{11}=u_{63}u_{15}u_5=u_{63}u_{15}u_7u_3,$$
$$u_{99}u_{53}=u_{127}u_{31}u_{63}u_{15}u_7=u_{127}u_{63}u_{31}u_9=u_{127}u_{63}u_{23}=u_{127}u_{41}=
u_{87}.$$
\end{example}

Now let us consider the case $p>2$.
We deduce easily from \eqref{mrules} that the subalgebra of $\overline G_{n+1}$ generated by $u_{2q-1}$ is the span
$$
\langle u_1, u_{2q\pm 1}, u_{4q\pm 1}, \ldots , u_{(p-1)q\pm 1}\rangle_{\Bbb Z};
$$ 
thus this subalgebra 
is a based subring of $\overline G_{n+1}$. Observe that \eqref{mrules} imply that as a based ring
the span $\langle u_1, u_{2q\pm 1}, u_{4q\pm 1}, \ldots , u_{(p-1)q\pm 1}\rangle_{\Bbb Z}$ does not depend
on $n$. We will denote this based ring by $K_p$. Thus $K_p$ has a basis $X_0, X_1, \ldots, X_{p-1}$
with multiplication determined by 
$$X_0=1,\; X_1X_i=X_{i-1}+X_i+X_{i+1}, 1\le i\le p-2,\; X_1X_{p-1}=X_{p-2}.$$
The identification of $K_p$ with the span $\langle u_1, u_{2q\pm 1}, u_{4q\pm 1}, \ldots , u_{(p-1)q\pm 1}\rangle_{\Bbb Z}$ is given by $X_0\mapsto u_1, X_1\mapsto u_{2q-1}, X_2\mapsto u_{2q+1}, X_3\mapsto u_{4q-1}$ etc.


By \eqref{mrules} for $1\le r<q$ we have
$$u_{2q-1}u_r=u_{2q-r}.$$
Using \eqref{mrules} again we get
$$(u_1+u_{2q-1}+u_{2q+1})u_r=u_{2q-1}^2u_r=u_{2q-1}u_{2q-r}=u_r+u_{2q-r}+u_{2q+r}$$
which implies 
$$u_{2q+1}u_r=u_{2q+r},\; (1\le r<q).$$

Similarly, we get
$$u_{4q-1}u_r=u_{4q-r},\; (1\le r<q),$$
$$u_{4q+1}u_r=u_{4q+r},\; (1\le r<q)$$
etc.
Thus we have proved the following

\begin{proposition} \label{next factor}
For $n\ge 1$ there is a decomposition of based rings 
$$\overline G_{n+1}=\overline G_n \otimes_{\Bbb Z} K_p.$$
\end{proposition}

\begin{corollary} \label{G factor}
For $n\ge 1$ we have an isomorphism of based rings:
$$\overline G_n=\overline G_1 \otimes K_p^{\otimes n-1}.$$
\end{corollary}
Recall that $\overline G_1={\rm Gr}(\Ver_p)$  is the Grothendieck ring of the Verlinde category.
Note also that this Corollary holds for all $p$, including $p=2$ as follows from Example \ref{p even}
(in this case $\overline G_1={\Bbb Z}$ and $K_2$ is the group ring ${\Bbb Z}[\Bbb Z/2]$ with its natural basis).

\begin{example} Here is a numerical example for Corollary \ref{G factor}. Let $p=5$ and consider
$u_{1023}\in \overline G_5$. We have
$$u_{1023}=u_{1249}u_{227}=u_{1249}u_{249}u_{23}=u_{1249}u_{249}u_{21}u_3.$$
Thus under the isomorphism of Corollary \ref{G factor} $u_{1023}$ corresponds to the element
$L_3\otimes X_4\otimes X_1\otimes X_1$.
\end{example}

\subsection{Categorifications of $K_p$} 
Let $\C_n$ be the semisimplification of the category $\Rep(\Bbb Z/{p^n})$.
Thus $\C_n$ is a symmetric fusion category with ${\rm Gr}(\C_n)=\overline G_n$. Proposition \ref{next factor}
implies that $\C_{n+1}\simeq \C_n\boxtimes \C_{n+1}^0$ where the category $\C_{n+1}^0$ satisfies
${\rm Gr}(\C_{n+1}^0)=K_p$. The goal of this section is to show that the categories $\C_{n}^0$ do not
depend on the value of $n\ge 2$.

More generally, we will show that the based ring $K_p$ has a unique
categorification by a symmetric fusion category $\D$ over a field of characteristic $p$. Here are two easy
cases:

(a) $p=2$: in this case the category $\D$ is pointed with underlying group $\Bbb Z/2$; thus possible categorifications
are given by homomorphisms $\Bbb Z/2\to {\bold k}^\times$, see e.g. \cite[Example 9.9.1]{EGNO}. Since $\mbox{char} {\bold k}=2$,
any such homomorphism is trivial and we have a unique categorification $\D={\rm Vec}_{\Bbb Z/2}$ (with trivial associator
and braiding).

(b) $p=3$: in this case $K_3$ has a basis $X_0, X_1, X_2$ with
$$X_0=1,\; X_1X_1=X_0+X_1+X_2,\; X_1X_2=X_1,\; X_2X_2=1$$
(thus $K_3\simeq {\rm Gr}(\Rep(S_3))$, where $\Rep(S_3)$ is the representation category of the symmetric group on three letters over the complex numbers). Let $\X_i\in \D, i=0,1,2$ be an object corresponding to $X_i\in K_p={\rm Gr}(\D)$.
It is clear that the dimension $\dim(\X_2)=1\in {\bold k}$ ($X_2^2=1$ implies that $\dim(\X_2)=\pm 1$ and
$\dim(\X_2)=-1$ would imply that $\dim(\X_1)=0$ which is impossible by \cite[Proposition 4.8.4]{EGNO}). 
Thus $\X_2$ generates a Tannakian subcategory of $\D$, equivalent to $\Rep(\Bbb Z/2)$. Then de-equivariantization
$\D_{\Bbb Z/2}$ of $\D$ with respect to this subcategory (see e.g. \cite[8.23]{EGNO}) is a fusion category of Frobenius-Perron dimension 3; hence pointed with underlying group $\Bbb Z/3$. Thus $\D_{\Bbb Z/2}={\rm Vec}_{\Bbb Z/3}$ with trivial
associator and braiding (since the only homomorphism $\Bbb Z/3\to {\bold k}^\times$ is the trivial one). Thus
the category $\D$ is the equivariantization $({\rm Vec}_{\Bbb Z/3})^{\Bbb Z/2}$ of the category ${\rm Vec}_{\Bbb Z/3}$ with respect to
an action of group $\Bbb Z/2$, see \cite[Theorem 8.23.3]{EGNO}. Note that the group $\Bbb Z/2$ must permute
simple objects of ${\rm Vec}_{\Bbb Z/3}$ nontrivially, since otherwise the category $({\rm Vec}_{\Bbb Z/3})^{\Bbb Z/2}$ is pointed.
Clearly, there is only one nontrivial homomorphism $\Bbb Z/2\to \Aut(\Bbb Z/3)$.
By the general theory (see \cite[Corollary 7.9]{ENOh} or \cite{Galindo}), actions of $\Bbb Z/2$ on ${\rm Vec}_{\Bbb Z/3}$ lifting
this homomorphism form a torsor over $H^2(\Bbb Z/2,\Bbb Z/3)=0$; thus there is only one such an action. 
Thus $\D \simeq ({\rm Vec}_{\Bbb Z/3})^{\Bbb Z/2}$ with respect to a unique action; in particular the category $\D$ is
unique up to a braided equivalence.

Assume now that $p>3$. Let $\D$ be a symmetric fusion category categorifying $K_p$ and let
$\X_i \in \D, i=0,\ldots ,p-1$ be the simple objects corresponding to $X_i\in K_p$.
Recall \cite[Theorem 1.5]{Ofib} that there exists a symmetric tensor functor $F: \D \to \Ver_p$
(this functor is known to be unique up to isomorphism, see \cite{EOV}). Let us determine
the homomorphism of based rings (i.e. sending basis elements of $K_p$ to positive integer combinations of basis elements of ${\rm Gr}(\Ver_p)$) $K_p={\rm Gr}(\D)\to {\rm Gr}(\Ver_p)$ induced by the functor $F$.

We have the following three homomorphisms $\phi_i: K_p\to {\rm Gr}(\Ver_p^+), i=1,2,3$:
$$\phi_1(X_1)=L_1+L_{p-2},\; \phi_2(X_1)=L_3,\; \phi_3(X_1)=-1,$$
and it is easy to see that these homomorphisms determine an isomorphism of $\BQ-$algebras 
$(K_p)_\BQ\simeq {\rm Gr}(\Ver_p^+)_\BQ \times {\rm Gr}(\Ver_p^+)_\BQ \times \BQ$ (where $A_\BQ :=A\otimes_{\Bbb Z} \BQ$).
Let us show that $\phi_1$ is the unique homomorphism of based rings $K_p\to {\rm Gr}(\Ver_p)$ (note that $\phi_2$ is not
a homomorphism of based rings since $\phi_2(X_{\frac{p-1}2})=0$).

\begin{proposition} \label{asexp}
Assume $p>3$.
Let $\D$ be a symmetric fusion category categorifying $K_p$ and let
$F: \D \to \Ver_p$ be a symmetric tensor functor. Then $F(\X_1)=L_1\oplus L_{p-2}$. In particular 
$\dim(\X_i)=(-1)^i$ and $S^2\X_1$ contains the unit object as a direct summand.
\end{proposition}

\begin{proof} 
It is easy to verify that there exists a ring homomorphism
$$K_p\;\to\;\mathbb{R},\quad X_i\mapsto [i+1]_q+[i]_q.$$
Since it sends the basis to positive numbers, we see that $\FP(X_1)=1+[2]_q$.
Thus we have the following options:

(1) $F(\X_1)=L_1\oplus L_{p-2}$,

(2) $F(\X_1)=L_{p-1}\oplus L_{p-2}$,

(3) $F(\X_1)=L_{p-1}\oplus L_{2}$,

(4) $F(\X_1)=L_1\oplus L_{2}$.

Options (2) and (3) are impossible since $F(\X_1)^{\otimes 2}$ would not contain $F(\X_1)$, but 
$X_1^2=X_0+X_1+X_2$. Option (4) would imply $\dim(\X_1)=3$ whence induction in $i$ gives 
$\dim(\X_i)=2i+1$ and $\dim(\X_{\frac{p-1}2})=p=0$ which is a contradiction, see e.g. \cite[Lemma 2.4.1]{BK}.
Thus the only possibility is option (1). This implies $\dim(\X_1)=-1$ and the formula for $\dim(\X_i)$ follows
by induction in $i$. Finally we see that $F(\wedge^2\X_1)=\wedge^2F(\X_1)=L_{p-2}\oplus L_3$, so
$\wedge^2\X_1$ does not contain the unit object as a summand. We get the last statement since
$\be$ is contained in $\X_1^{\otimes 2}=S^2\X_1\oplus \wedge^2\X_1$.
\end{proof}

Let $\delta \in {\bold k}$ and let $\uRep_0(O(\delta),{\bold k})$ be the {\em Brauer category} as defined in
\cite[9.3]{DeSymm} and let $\uRep(O(\delta),{\bold k})$ be its Karoubian envelope; we will denote
by $\bullet$ the generating object of $\uRep(O(\delta),{\bold k})$. Recall (see \cite[9.4]{DeSymm})
that the category $\uRep(O(\delta),{\bold k})$ has the following universal property.

\begin{proposition}\label{(c)} Let $\cA$ be a Karoubian symmetric pseudo-tensor category.
The functor $G\mapsto G(\bullet)$ gives an equivalence of the following categories:

(i) ${\bold k}-$linear symmetric monoidal functors $G: {\rm Rep}(O(\delta),{\bold k})\to \cA$; and

(ii) objects of $\cA$ of dimension $\delta$ equipped with non-degenerate symmetric bilinear form
and invertible morphisms thereof.
\end{proposition} 

Let $\uRep^{ss}(O(\delta),{\bold k})$ denote the quotient of the category $\uRep(O(\delta),{\bold k})$ by 
the negligible morphisms, see e.g. \cite[6.1]{DeSymm} (the category $\uRep^{ss}(O(\delta),{\bold k})$ is
sometimes called the {\em semisimplification} of $\uRep(O(\delta),{\bold k})$; however see Remark \ref{semise}
below). For typographic reasons, we do not employ the notation $\mathcal{C}\mapsto \overline{\mathcal{C}}$ we used for the semisimplificiation of tensor categories in earlier sections.

\begin{theorem} \label{delta-1}
Assume that $p>3$. Then the category ${\rm Rep}^{ss}(O(-1),{\bold k})$ is semisimple and
we have an isomorphism of based rings $${\rm Gr}({\rm Rep}^{ss}(O(-1),{\bold k}))\cong K_p\otimes {\Bbb Z}[\Bbb Z/2].$$
\end{theorem}

\begin{remark}\label{semise}
 (i) If $\mbox{char} {\bold k}=0$, all the categories $\uRep^{ss}(O(\delta),{\bold k})$ are semisimple
and were computed in \cite[Th\'eor\`em 9.6]{DeSymm}.

(ii) If $\mbox{char} {\bold k} >0$ then the category $\uRep^{ss}(O(\delta),{\bold k})$ is semisimple if and only if
$\delta$ is an element of the prime subfield $\BF_p\subset {\bold k}$. The semisimplicity implies $\delta \in \BF_p$
by \cite[Lemma 2.2]{EHO}. Conversely, if $\delta \in \BF_p$, then using the universal property one
constructs a monoidal functor $\uRep(O(\delta),{\bold k})\to {\rm Vec}$ sending the generating object $\bullet$ to
a vector space of dimension $n\equiv \delta \pmod{p}$. Then the semisimplicity of $\uRep^{ss}(O(\delta),{\bold k})$
follows by \cite{AK2}.

(iii) Let $p>2$, $\delta \in \BF_p$. Let $n\in {\Bbb Z}_{\ge 0}$ satisfy $n\equiv \delta \pmod{p}$. Let $\Rep(O(n))$ be
the category of representations of the orthogonal group $O(n)$ (considered as an algebraic group) over the field ${\bold k}$. By the universal
property above, there is a symmetric tensor functor $\uRep(O(\delta),{\bold k})\to \Rep(O(n))$ sending the 
object $\bullet$ to the tautological representation of $O(n)$. The characteristic free first fundamental 
theorem of invariant theory (proved in \cite{dCP}) states that this functor is full, i.e. surjective on Hom's. 
It follows that $\uRep^{ss}(O(\delta),{\bold k})$ equals the semisimplification of the image in $\Rep(O(n))$ of this
functor. This image is contained in the subcategory $\T (O(n))=\mathrm{Tilt}_p(O(n))$ of $\Rep(O(n))$ formed by {\em tilting modules}
(we say that an $O(n)-$module is tilting if its restriction to the connected subgroup $SO(n)$ is tilting). Thus
$\uRep^{ss}(O(\delta),{\bold k})$ is equivalent to the subcategory of the semisimplification $\T^{ss}(O(n))$ generated
by the image of the tautological representation. One can show that
if we choose $n$ to satisfy $0\le n<p$, then the functor $\uRep^{ss}(O(\delta),{\bold k})\to \T^{ss}(O(n))$ is surjective
and hence we have an equivalence $\uRep^{ss}(O(\delta),{\bold k})\simeq \T^{ss}(O(n))$ (we will prove this for $\delta =-1$
below). This implies that $\uRep^{ss}(O(\delta),{\bold k})$ has just finitely many simple objects for $\delta \ne 2$
(so it is a fusion category); however for $\delta =2$ we have 
$\uRep^{ss}(O(\delta),{\bold k})\simeq \T^{ss}(O(2))=\Rep(O(2))$ and this category has infinitely many simple objects.
In particular, Theorem \ref{delta-1} does fail for $p=3$.

(iv) For any $p$ (including $p=2$) we have $\uRep^{ss}(O(0),{\bold k})={\rm Vec}$ and $\uRep^{ss}(O(1),{\bold k})={\rm Vec}_{\Bbb Z/2}$.
\end{remark}

\begin{proof}[Proof of Theorem~\ref{delta-1}] By Remark \ref{semise} (iii) we already know that the category $\uRep^{ss}(O(-1),{\bold k})$ is semisimple
and has finitely many simple objects, so it remains to compute its Grothendieck ring.

We have a canonical functor $\uRep(O(-1),{\bold k})\to \uRep^{ss}(O(-1),{\bold k})$. By Remark \ref{semise}
(iii) we have an injective tensor functor $\uRep^{ss}(O(-1),{\bold k})\to \T^{ss}(O(p-1))$. Let $\T^{ss}(O(p-1))\to \Ver_p$ be
the fiber functor. Using \cite[4.3.3]{Ofib} we see that the composition 
$$\uRep(O(-1),{\bold k})\to \uRep^{ss}(O(-1),{\bold k})\to \T^{ss}(O(p-1))\to \Ver_p$$
sends the standard object $\bullet$ to $L_1\oplus L_{p-2}$ (in other words the principal nilpotent element acts
on the tautological representation of $O(p-1)$ with Jordan blocks of sizes $p-2$ and 1).

Now let $\D$ be a symmetric fusion category with ${\rm Gr}(\D)=K_p$ and let $\D'=\D \boxtimes {\rm Vec}_{\Bbb Z/2}$.
The object $\X_1\boxtimes a \in \D'$ where $a\in {\rm Vec}_{\Bbb Z/2}$ is the nontrivial simple object has a symmetric bilinear form (see Proposition \ref{asexp}); thus by the
universal property we have a tensor functor $\uRep(O(-1),{\bold k})\to \D'$ sending the object $\bullet$ to
$\X_1\boxtimes a$. Composing it with the fiber functor $\D' \to \D \to \Ver_p$ we see that the composition 
sends the object $\bullet$ to $L_1\oplus L_{p-2}\in \Ver_p$. 

It is obvious that $L_1\oplus L_{p-2}\in \Ver_p$ has a unique up to isomorphism non-degenerate symmetric
bilinear form. Thus it follows from the universal property that the above two compositions are isomorphic
tensor functors. In particular, these compositions annihilate the same morphisms. Since the functor
$\D'\to \Ver_p$ is faithful, we see that the functor $\uRep(O(-1),{\bold k})\to \D'$ annihilates all the negligible
morphisms. Hence this functor factors through the functor 
$$
\uRep(O(-1),{\bold k})\to \uRep^{ss}(O(-1),{\bold k})
$$ 
and we have a symmetric tensor functor $\uRep^{ss}(O(-1),{\bold k})\to \D'$ sending the standard object $\bullet$
of $\uRep^{ss}(O(-1),{\bold k})$ to $\X_1\boxtimes a$. The object $\X_1\boxtimes a$ generates $\D'$, so this
functor is surjective. On the other hand, Lemma \ref{levrank} shows that
$$\FP(\uRep^{ss}(O(-1),{\bold k}))\le \FP(\T^{ss}(O(p-1)))=\FP(\D').$$
Thus the surjective functor $\uRep^{ss}(O(-1),{\bold k})\to \D'$ (as well as injective functor $\uRep^{ss}(O(-1),{\bold k})\to \T^{ss}(O(p-1)))$ is an equivalence by \cite[Proposition 6.3.4]{EGNO} and the proof is complete.
\end{proof} 

\begin{lemma}\label{levrank}
Assume $p>3$. We have $\FP(\T^{ss}(O(p-1)))=\FP(\D').$
\end{lemma}

\begin{proof} Let $\T(SO(p-1))$ be the category of tilting modules over group $SO(p-1)$ and let 
$\T^{ss}(SO(p-1))$ be its semisimplification. The exact sequence of algebraic groups
$$1\to SO(p-1)\to O(p-1)\to \Bbb Z/2\to 1$$
shows that the category $\T(O(p-1))$ can be obtained from the category $\T(SO(p-1))$ via equivariantization with
respect to $\Bbb Z/2-$action, $\T(O(p-1))=\T(SO(p-1))^{\Bbb Z/2}$ see e.g. \cite[4.15]{EGNO}. It follows
(see \cite[Proposition 3.15]{EO}) that $\T^{ss}(O(p-1))=\T^{ss}(SO(p-1))^{\Bbb Z/2}$. Hence
$$\FP(\T^{ss}(O(p-1)))=2\FP(\T^{ss}(SO(p-1))).$$

Now let $\T({\rm Spin}(p-1))$ be the category of tilting modules over the spin group ${\rm Spin}(p-1)$ and let 
$\T^{ss}({\rm Spin}(p-1))$ be its semisimplification. The exact sequence of algebraic groups
$$1\to \Bbb Z/2\to {\rm Spin}(p-1)\to SO(p-1)\to 1$$
determines a $\Bbb Z/2-$grading on the category $\T({\rm Spin}(p-1))$ with neutral component $\T(SO(p-1))$;
hence $\T^{ss}(SO(p-1))$ is a neutral component of $\Bbb Z/2-$grading on $\T^{ss}({\rm Spin}(p-1))$. Thus
$$\FP(\T^{ss}(O(p-1)))=2\FP(\T^{ss}(SO(p-1)))
$$
$$
=\FP(\T^{ss}({\rm Spin}(p-1))).
$$

According to \cite{GM} the fusion rules (and hence Frobenius-Perron dimension) of the category
$\T^{ss}({\rm Spin}(p-1))$ coincide with those of the $\BC-$linear fusion category $\C(\mathfrak{so}(p-1),3)$ associated with affine Lie algebra $\mathfrak{so}(p-1)$ at level $3$. Also it is well known that the fusion rules of the category $\D$ coincide with 
the fusion rules of the even part of $\C(\mathfrak{sl}(2),2p-2)$; in particular
$\FP(\D')=\FP(\C(\mathfrak{sl}(2),2p-2))$. Thus the statement of the Lemma is equivalent to
the following:
$$\FP(\C(\mathfrak{so}(p-1),3))=\FP(\C(\mathfrak{sl}(2),2p-2)).$$
Equation \eqref{levrank}  is a manifestation of the {\em level rank duality} connecting affine Lie algebras
$\mathfrak{so}(p-1)$ at level 3 and $\mathfrak{so}(3)$ at level $p-1$ (which, in view of standard conventions, identifies with
$\mathfrak{sl}(2)$ at level $2p-2$). It can be reduced to a trigonometric identity using \cite[Theorem 3.3.20]{BK};
we refer the reader to \cite[Corollary 1.9]{OW} for a nice elementary proof of this identity (use formula (i) 
with $U'=\{ \frac12\}$ and take into account \cite[Remark 1.10]{OW}).
\end{proof}

\begin{remark} A more natural approach to the proof of Lemma \ref{levrank} would be to use a positive characteristic version of the level-rank duality which we discussed in Section 4. This duality connects the category $\T^{ss}(O(p-1))$ with the semisimplification $\T^{ss}(SOSp(1|2))$ of the category of tilting
modules over supergroup $SOSp(1|2)$ of orthogonal transformations of the superspace of superdimension $1|2$. However, the theory of tilting modules over  $SOSp(1|2)$ is not well documented, so we chose to invoke the better documented characteristic zero results. 
\end{remark}

The proof of Theorem \ref{delta-1} yields 

\begin{corollary} Let $\D$ be a symmetric fusion category such that ${\rm Gr}(\D)=K_p$. Then $\D$ is equivalent
to the neutral component of the standard $\Bbb Z/2-$grading of the category ${\rm Rep}^{ss}(O(-1),{\bold k})$.
\end{corollary}

Recall that $\C_n$ is the semisimplification of the category $\Rep(\Bbb Z/{p^n})$. The following can be seen as an example of Conjecture~\ref{mai}.

\begin{corollary} We have an equivalence of symmetric fusion categories
$$\C_n\simeq \Ver_p\boxtimes \D^{\boxtimes n-1},$$
where
$$\D\;=\;\begin{cases} \mathrm{Vec}_{\Z/2}&\quad\mbox{if $p=2$},\\
\left(\mathrm{Vec}_{\Z/3}\right)^{\Z/2}&\quad\mbox{if $p=3$},\\
\left(\Ver_p^+(\mathrm{Spin}(p-1))\right)^{\Z/2}&\quad\mbox{if $p>3$.}
\end{cases}$$
\end{corollary}

\begin{remark}The same result is valid for the representation category 
of the Hopf algebra $\bold k[x]/(x^{p^n})$ where $x$ is a primitive element, and 
more generally for any cocommutative Hopf algebra structure on 
$\bold k[x]/(x^{p^n})$ (there are many such structures related to truncations of 1-dimensional formal group laws). Indeed, the analysis of the Green ring in Sections 5.2, 5.4 and 5.5 of \cite{B1} reproduced in this section works for any cocommutative coproduct, so the Grothendieck ring of the semisimplification is independent on the Hopf structure. But we show that
in this case the Grothendieck ring determines the category. 
\end{remark}


\begin{thebibliography}{999999}
\bibitem[AK1]{AK} Y. Andr\'e, B. Kahn, with an appendix by P. O'Sullivan, 
{ Nilpotence, radicaux et structures mono\"idales}, arXiv:math/0203273,
Rendiconti del Seminario Matematico dell'Universita di Padova 108 (2002), 107--291. 

\bibitem[AK2]{AK2} Y.~Andr\'e, B.~Kahn, { Erratum: ``Nilpotency, radicals and monoidal structures'' [Rend. Sem. Mat. Univ. Padova 108 (2002), 107--291; MR1956434]}.  Rend. Sem. Mat. Univ. Padova 113 (2005), 125--128. 

\bibitem[BK]{BK}
  { B.~Bakalov, A.~Kirillov Jr}, 
  { Lectures on Tensor Categories and Modular Functors}, 
  University Lecture Series. American Mathematical Society volume 21, Providence, RI, (2001).
\bibitem[B1]{B1} D. Benson, Commutative Banach algebras and modular representation theory, Memoirs of the AMS, to appear, arXiv:2008.13155.
\bibitem[B2]{B2} D. Benson, Some conjectures and their consequences for tensor products of modules over a finite $p$-group, {J. Algebra},
Volume 558, 2020, p. 24--42.
\bibitem[BS]{BS} D. Benson, P. Symonds, 
The non-projective part of the tensor powers of a module,
 J. Lond. Math. Soc. (2) Volume 101, Issue 2, 2020, Pages 828--856.
\bibitem[Bi]{Bi} P. Biane, Estimation asymptotique des multiplicites dans les puissances tensorielles d'un g-module,
C. R. Acad. Sci. Paris S'erie. I Math. 316 (1993), no. 8, 849--852.
\bibitem[BEEO]{BEEO} J. Brundan, I. Entova-Aizenbud, P. Etingof, V. Ostrik, Semisimplification of the category of tilting modules for $GL_n$, Advances Math. 375 (2020), 107331,  	arXiv:2002.01900.
\bibitem[CM]{CM} D. Collingwood, W. McGovern, Nilpotent orbits in semisimple Lie algebras, 1993.
\bibitem[Co]{Co} M. J. Collins, On Jordan's theorem for complex linear groups, Journal of Group Theory, 10 (4): 411--423.
\bibitem[CEO]{CEO} K. Coulembier, P. Etingof, V. Ostrik, On Frobenius exact symmetric tensor categories. With an appendix by A. Kleshchev. Ann. of Math. (2) 197 (2023), no. 3, 1235--1279.
\bibitem[COT]{COT} K. Coulembier, V. Ostrik, D. Tubbenhauer, Growth rates of the number of indecomposable
summands in tensor powers. To appear in Alg. Rep. Th. 10.1007/s10468-023-10245-7 arXiv:2301.00885
\bibitem[dCP]{dCP} C.~de Concini, C.~Procesi, {A characteristic free approach to invariant theory,}
Adv. Math. \textbf{21} (1976), no. 3, 330--354.
\bibitem[D1]{D1} P. Deligne,  Cat\'egories tensorielles, Mosc. Math. J., 2 (2002), no. 2, 227--248.
\bibitem[D2]{DeSymm} P.~Deligne, {La cat\'egorie des repr\'esentations du groupe sym\'etrique 
$S_t$, lorsque $t$ n'est pas un entier naturel,} in: Algebraic Groups and Homogeneous Spaces, in: Tata Inst. Fund. Res. Stud. Math., Tata Inst. Fund. Res., Mumbai, 2007,  209--273.
\bibitem[E]{E} P. Etingof, Koszul duality and the PBW theorem in symmetric tensor categories in positive characteristic, arXiv:1603.08133, Advances in Mathematics
Volume 327, Pages 128--160. 
\bibitem[EGNO]{EGNO}P.~Etingof, S.~Gelaki, D.~Nikshych, V.~Ostrik, 
Tensor categories. 
Mathematical Surveys and Monographs, 205. American Mathematical Society, Providence, RI, 2015. 
\bibitem[EHO]{EHO} P.~Etingof, N.~Harman, V.~Ostrik, { $p-$adic dimensions in symmetric tensor categories in characteristic $p$,} Quantum Topology {\bf 9} (2018), no. 1, p. 119-140.
\bibitem[EK]{EK} P. Etingof, A. Kannan, Lectures on symmetric tensor categories, arXiv:2103.04878.
\bibitem[EKO]{EKO} P. Etingof, A. Kannan, V. Ostrik, Lie theory in the Verlinde category and weak Jordan algebras, to appear. 
\bibitem[ENO]{ENOh} P.~Etingof, D.~Nikshych, V.~Ostrik { Fusion categories and homotopy theory,} 
Quantum Topology {\bf 1} (2010), no. 3, p. 209--273.
\bibitem[EO]{EO} P. Etingof, V. Ostrik, On semisimplification of tensor categories, arXiv:1801.04409, {\rm in:} 
V. Baranovsky et al. (eds.), Representation Theory and Algebraic Geometry, Trends
in Mathematics, p. 3-35. 
\bibitem[EOV]{EOV} P.~Etingof, V.~Ostrik, S.~Venkatesh { Computations in symmetric fusion categories in characteristic $p$,} 
Int. Math. Res. Not. {\bf 2017}, no. 2, p. 468-489.
\bibitem[Ga]{Galindo} C.~Galindo {Clifford theory for tensor categories,}
 J. Lond. Math. Soc. (2) 83 (2011), no. 1, 57-78. 
 \bibitem[GK]{GK} S. Gelfand and D. Kazhdan. Examples of tensor categories, \textit{Invent. Math.} 109(1992),no. 3, 595--617.
\bibitem[GM]{GM} G. Georgiev and O. Mathieu, Fusion rings for modular representations of Chevalley groups, {Contemp. Math}.175(1994), 89--100.
\bibitem[G]{Gr} J.~A.~Green, {The modular representation algebra of a finite group,}
Illinois J. Math. \textbf{6} (1962), 607-619.
\bibitem[J]{J} 
J.~C.~Jantzen,
{Representations of algebraic groups}, 2nd edition, American Mathematical Society,
Providence, RI, 2003. 
\bibitem[K]{K} A. Kannan, New Constructions of Exceptional Simple Lie Superalgebras with Integer Cartan Matrix in Characteristics 3 and 5 via Tensor Categories, arXiv:2108.05847, Transformation groups, 2022. 
\bibitem[O]{Ofib} V. Ostrik, On symmetric fusion categories in positive characteristic, 
\textit{Selecta Mathematica}, v.26, 2020, arxiv:1503.01492.
\bibitem[Op]{Op} E. M. Opdam, Some applications of hypergeometric shift operators.
Invent. Math. 98 (1989), no. 1, 1--18.
\bibitem[OW]{OW} W.~M.~Oxbury, S.~M.~J.~Wilson, Reciprocity laws in the {V}erlinde formulae for the classical
              groups, 
{Trans. Amer. Math. Soc.}, {\bf 348} (1996), no. 7, 2689--2710.
\bibitem[PR]{PR} O. Postnova, N. Reshetikhin, On multiplicities of irreducibles in large tensor product
of representations of simple Lie algebras, Letters Math. Phys. Ser. I, 110 (2020). 
\bibitem[RS]{RS} V. Rittenberg, M. Scheunert, A Remarkable Connection Between the Representations of the Lie Superalgebras $osp(1, 2n)$ and the Lie Algebras $o(2n + 1)$,  Communications in Mathematical Physics 83(1), p.1--9, 1982.
\bibitem[Ro1]{Ro1} G. Robinson, Bounding the size of permutation groups and complex linear groups of odd order, Journal of Algebra 335 (2011) 163--170.
\bibitem[Ro2]{Ro2} G. Robinson, 
Bounding the order of complex linear groups and permutation groups with selected composition factors
{arXiv:2303.06011}
\bibitem[St]{St} D. Stewart, On the minimal modules for exceptional Lie algebras: Jordan blocks and stabilisers, arXiv:1508.02918, LMS J. Comput. Math. 19 (2016) 235--258. 
\bibitem[Su]{Su}  R.~Suter, Coxeter and dual Coxeter numbers. Comm. Algebra 26 (1998), no. 1, 147--153.
\bibitem[TZ]{TZ} T. Tate, S. Zelditch, Lattice path combinatorics and asymptotics of multiplicities of weights in tensor powers, J. Funct. Anal., Volume 217, Issue 2, 2004, Pages 402--447.
\bibitem[V]{V} S. Venkatesh, Harish-Chandra pairs in the Verlinde category in positive characteristic, 
arXiv:1909.11240. 
\end{thebibliography}
\end{document}